\input ./style/arxiv-general.cfg
\documentclass[aop,MSNbibl,seceqn,dvips]{arximspdf}
\makeatletter
   \@ifpackageloaded{graphicx}{}{\usepackage{graphicx}}
\makeatother

\doi{10.1214/13-AOP883} 
\volume{43}
\issue{3}
\pubyear{2015}
\firstpage{1121}
\lastpage{1156}

\makeatletter

\newcommand{\rrVert}{\Vert}
\newcommand{\rrvert}{\vert}
\newcommand{\llVert}{\Vert}
\newcommand{\llvert}{\vert}
\newproclaim{example}[theorem]{Example}
\newproclaim{definition}{Definition}
\newtheorem{theorem}{Theorem}[section]
\newtheorem{condition}[theorem]{Condition}
\newtheorem{lemma}[theorem]{Lemma}
\newproclaim{remark}[theorem]{Remark}
\makeatother

\begin{document}
\begin{frontmatter}

\title{On the large deviation rate function for the empirical measures
of reversible
jump Markov~processes}
\runtitle{LDP for empirical measures}

\begin{aug}
\author[A]{\fnms{Paul} \snm{Dupuis}\thanksref{T1}\ead[label=e1]{dupuis@dam.brown.edu}}
\and
\author[B]{\fnms{Yufei} \snm{Liu}\corref{}\thanksref{T2}\ead[label=e2]{liuyf03@gmail.com}}
\thankstext{T1}{Supported in part by the Department of Energy
(DE-SCOO02413), the NSF (DMS-10-08331), and the Army
Research Office (W911NF-09-1-0155, W911NF-12-1-0222).}
\thankstext{T2}{Supported in part by the Department of Energy
(DE-SCOO02413).}
\runauthor{P. Dupuis and Y. Liu}
\affiliation{Brown University}
\address[A]{Division of Applied Mathematics\\
Brown University\\
Box F\\
182 George St.\\
Providence, Rhode Island 02912\\
USA\\
\printead{e1}} 
\address[B]{Google Inc.\\
1000 Amphitheatre Parkway\\
Mountain View, California 94043\\
USA\\
\printead{e2}}
\end{aug}

\received{\smonth{2} \syear{2013}}
\revised{\smonth{9} \syear{2013}}

%
\begin{abstract}
The large deviations principle for the empirical measure for both continuous
and discrete time Markov processes is well known. Various expressions are
available for the rate function, but these expressions are usually as the
solution to a variational problem, and in this sense not explicit. An
interesting class of continuous time, reversible processes was
identified in
the original work of Donsker and Varadhan for which an explicit
expression is
possible. While this class includes many (reversible) processes of interest,
it excludes the case of continuous time pure jump processes, such as a
reversible finite state Markov chain. In this paper, we study the large
deviations principle for the empirical measure of pure jump Markov processes
and provide an explicit formula of the rate function under reversibility.
\end{abstract}


\end{frontmatter}
%
\section{Introduction}
\label{intro}

Let $X ( t ) $ be a time homogeneous Markov process with Polish
state space $S$, and let $P ( t,x,dy ) $ be the transition function
of $X ( t ) $. For $t\in{}[0,\infty)$, define $T_{t}$ by
\[
T_{t}f ( x ) \doteq\int_{S}f ( y ) P ( t,x,dy ).
\]
Then $T_{t}$ is a contraction semigroup on the Banach space of bounded, Borel
measurable functions on $S$ (\cite{ethkur}, Chapter~4.1). We use $\mathcal{L}$
to denote the infinitesimal generator of $T_{t}$ and $\mathcal{D}$ the domain
of $\mathcal{L}$ (see \cite{ethkur}, Chapter~1). Hence, for each bounded
measurable function $f\in\mathcal{D}$,
\[
\mathcal{L}f ( x ) =\lim_{t\downarrow0}\frac{1}{t} \biggl[ \int
_{S}f ( y ) P ( t,x,dy ) -f ( x ) \biggr].
\]
The empirical measure (or normalized occupation measure) up to time $T$
of the
Markov process $X ( t ) $ is defined by
%
\begin{equation}
\eta_{T} ( \cdot ) \doteq\frac{1}{T}\int_{0}^{T}
\delta_{X (
t ) } ( \cdot ) \,dt. \label{eta}%
\end{equation}
Let $\mathcal{P} ( S ) $ be the metric space of probability
measures on $S$ equipped with the L\'evy--Prohorov metric, which is compatible
with the topology of weak convergence. For $\eta\in\mathcal{P} (
S ) $, define
%
\begin{equation}
I ( \eta ) \doteq-\mathop{\inf_{u\in\mathcal{D}}}_{u>0}\int
_{S}\frac{\mathcal{L}u}{u}\,d\eta. \label{rate}%
\end{equation}
It is easy to check that $I$ thus defined is lower semicontinuous under the
topology of weak convergence. Consider the following regularity assumption.

\begin{condition}
\label{cond_diff} There exists a probability measure $\lambda$ on $S$ such
that for $t>0$ the transition functions $P ( t,x,dy ) $ have
densities with respect to $\lambda$, that is,
%
\begin{equation}
P ( t,x,dy ) =p ( t,x,y ) \lambda ( dy ). \label{eq_diff}%
\end{equation}
\end{condition}

Under additional recurrence and transitivity conditions, Donsker and\break Varadhan
\cite{donvar1,donvar3} prove the following. For any open set $O\subset
\mathcal{P} ( S ) $
%
\begin{equation}
\liminf_{T\rightarrow\infty}\frac{1}{T}\log P \bigl( \eta_{T}
( \cdot ) \in O \bigr) \geq-\inf_{\eta\in O}I ( \eta ),
\label{ldl}%
\end{equation}
and for any closed set $C\subset\mathcal{P} ( S ) $
%
\begin{equation}
\limsup_{T\rightarrow\infty}\frac{1}{T}\log P \bigl( \eta_{T}
( \cdot ) \in C \bigr) \leq-\inf_{\eta\in C}I ( \eta ).
\label{ldu}%
\end{equation}

We refer to (\ref{ldl}) as the large deviation lower bound and (\ref
{ldu}) as
the large deviation upper bound. Under ergodicity, the empirical measure
$\eta_{T}$ converges to the invariant distribution of the Markov process
$X ( t ) $. The large deviation principle characterizes this
convergence through the associated rate function. While there are many
situations where an explicit formula for (\ref{rate}) would be useful,
it is
in general difficult to solve the variational problem. The main existing
results on this issue are for the self-adjoint case in the continuous time
setting; see \cite{donvar1,pin,str1}. Specifically, suppose there is a
$\sigma$-finite measure $\varphi$ on $S$, and that the densities in~(\ref{eq_diff}) satisfy the following reversibility condition:
%
\begin{equation}
p ( t,x,y ) =p ( t,y,x ) \qquad\mbox{almost everywhere } ( \varphi\times\varphi ).
\label{rev}%
\end{equation}
Then $T_{t}$ is self-adjoint. If we denote the closure of $\mathcal{L}$ by
$\bar{\mathcal{L}}$ (see, e.g., \cite{ethkur}, page~16) and the domain of
$\bar{\mathcal{L}}$ by $\mathcal{D} ( \bar{\mathcal{L}} ) $, then
$\bar{\mathcal{L}}$ is self-adjoint and negative semidefinite (since $T_{t}$
is a contraction). We denote by $(-\bar{\mathcal{L})}^{1/2}$ the canonical
positive semidefinite square root of $-\bar{\mathcal{L}}$ (\cite{rud},
Chapter~12). Let $\bar{\mathcal{D}}_{1/2}$ be the domain of $(-\bar{\mathcal{L}})^{1/2}$. Donsker and Varadhan \cite{donvar1}, Theorem~5, show under certain
conditions that $I$ defined by (\ref{rate}) has the following properties:
$I ( \mu ) <\infty$ if and only if $\mu\ll\varphi$ and $ (
d\mu/d\varphi ) ^{1/2}\in\bar{\mathcal{D}}_{1/2}$, and with
$f\doteq
d\mu/d\varphi$ and $g\doteq f^{1/2}$,
%
\begin{equation}
I ( \mu ) =\bigl\llVert (-\bar{\mathcal{L}})^{1/2}g\bigr\rrVert
^{2}, \label{dvr}%
\end{equation}
where $\llVert \cdot\rrVert $ denotes the $L^{2}$ norm with
respect to
$\varphi$. Typically, $\varphi$ is taken to be the invariant
distribution of
the process.

It should be noted that this explicit formula does not apply to one of the
simplest Markov processes, namely, continuous time Markov jump
processes with
bounded infinitesimal generators. Let $\mathcal{B} ( S ) $ be the
Borel $\sigma$-algebra on $S$ and let $\alpha ( x,\Gamma ) $
be a
transition kernel on $S\times\mathcal{B} ( S ) $. Let $B (
S ) $ denote the space of bounded Borel measurable functions on
$S$ and
let $q\in B ( S ) $ be nonnegative. Then
%
\begin{equation}
\mathcal{L}f ( x ) \doteq q ( x ) \int_{S} \bigl( f ( y )
-f ( x ) \bigr) \alpha ( x,dy ) \label{generator}%
\end{equation}
defines a bounded linear operator on $B ( S ) $ and $\mathcal{L}$
is the generator of a Markov process that can be constructed as
follows. Let
$ \{ X_{n},n\in%
\mathbb{N}
\} $ be a Markov chain in $S$ with transition probability
$\alpha ( x,\Gamma ) $, that is,
%
\begin{equation}
P ( X_{n+1}\in\Gamma|X_{0},X_{1},
\ldots,X_{n} ) =\alpha ( X_{n},\Gamma ) \label{transition}%
\end{equation}
for all $\Gamma\in\mathcal{B} ( S ) $ and $n\in%
\mathbb{N}
$. Let $\tau_{1},\tau_{2},\ldots$ be independent and exponentially distributed
with mean $1$, and independent of $ \{ X_{n},n\in%
\mathbb{N}
\} $. Define a sojourn time $s_{i}$ for each $i=1,2,\ldots$ by
%
\begin{equation}
q ( X_{i-1} ) s_{i}=\tau_{i}.
\label{soj}%
\end{equation}
Then
\[
X ( t ) =X_{n}\qquad\mbox{for }\sum_{i=1}^{n}s_{i}
\leq t<\sum_{i=1}^{n+1}s_{i}%
\]
(with the convention $%
{ \sum_{i=1}^{0}}
s_{i}=0$) defines a Markov process $ \{ X ( t ),t\in
{}[0,\infty) \} $ with infinitesimal generator $\mathcal{L}$,
and we
call this process a Markov jump process.

A very simple special case is as follows. Using the notation above, assume
$S= [ 0,1 ] $, $q\equiv1$ and for each $x\in [ 0,1 ] $,
$\alpha ( x,\cdot ) $ is the uniform distribution on $ [
0,1 ] $. The infinitesimal generator $\mathcal{L}$ defined in
(\ref{generator}) reduces to
\[
\mathcal{L}f ( x ) =\int_{0}^{1}f ( y ) \,dy-f ( x
),
\]
which is clearly self-adjoint with respect to Lebesgue measure. If $C$
is the
collection of all Dirac measures on $S$, then $C$ is closed under the topology
of weak convergence on $\mathcal{P} ( S ) $. Hence, a large
deviation upper bound would imply
%
\begin{equation}
\limsup_{T\rightarrow\infty}\frac{1}{T}\log P ( \eta_{T}\in C
) \leq-\inf_{\mu\in C}I ( \mu ). \label{eq1}%
\end{equation}
However, the probability that the very first exponential holding time is
bigger than $T$ is exactly $\exp \{ -T \} $, and when this happens,
the empirical measure is a Dirac measure located at some point that is
uniformly distributed on $ [ 0,1 ] $. Hence,
\[
\liminf_{T\rightarrow\infty}\frac{1}{T}\log P \bigl( \eta_{T}
( \cdot ) \in C \bigr) \geq\liminf_{T\rightarrow\infty}\frac
{1}{T}\log P
( \tau_{1}>T ) =-1.
\]
In fact, we will prove later that the rate function for the empirical measure
of this Markov jump process never exceeds $1$. However, if the upper bound
held with the function defined in (\ref{dvr}), one would have $I (
\delta_{a} ) =\infty$ for $a\in [ 0,1 ] $, and by
(\ref{eq1})
\[
\limsup_{T\rightarrow\infty}\frac{1}{T}\log P \bigl( \eta_{T}
( \cdot ) \in D \bigr) =-\infty,
\]
which is impossible.

This example shows that this type of Markov jump process is not covered by
\cite{donvar1,donvar3}. In fact, the transition function $P (
t,x,dy ) $ takes the form
\[
P ( t,x,dy ) =e^{-t}\delta_{x} ( dy ) + \bigl(
1-e^{-t} \bigr) 1_{ [ 0,1 ] } ( y ) \,dy,
\]
which means that we cannot find a reference probability measure $\lambda
$ on
$S$ such that $P ( t,x,\cdot ) $ has a density with respect to
$\lambda ( \cdot ) $ for almost all $x\in S$ and $t>0$, which
is a
violation to Condition \ref{cond_diff} used in \cite{donvar1,donvar3}, and
also violates the form of reversibility needed for (\ref{dvr}).

A condition such as Condition \ref{cond_diff} holds naturally for Markov
processes that possess a ``diffusive'' term
in the dynamics, which is not the case for Markov jump processes, and
the form
of the rate function given in (\ref{dvr}) will not be valid for these
processes either. The purpose of the current paper is to establish a large
deviation principle for the empirical measures of reversible Markov jump
processes, and to provide an explicit formula for the rate function
like the
one given in (\ref{dvr}). We also show why the boundedness of the rate
function results from the fact that tilting of the exponential holding times
with bounded relative entropy cost can be used for target measures that are
not absolutely continuous with respect to the invariant distribution.

Finally, we mention that \cite{chelu} evaluates (\ref{rate}) for certain
classes of measures when $\mathcal{L}$ is the generator of a jump Markov
process satisfying various conditions. However, it does not present an
expression for an arbitrary measure, and indeed in appears that the authors
are unaware that (\ref{dvr}) is not the correct rate function for such
processes, or that the large deviation principle had not been established.

The paper is organized as follows. Section~\ref{setup} presents our
assumptions on the process. In Section~\ref{state_ldp}, we state the main
result, Theorem~\ref{large dev}. The proof of Theorem~\ref{large dev} is
divided into two sections, Section~\ref{lpub} for the upper bound and
Section~\ref{lplb} for the lower bound. In the final section, we
discuss the special
feature of Markov jump processes that leads to the boundedness of the
rate function.

\section{Assumptions}
\label{setup}

Our first assumption is that the Polish state space $S$ is compact. While
compactness is not needed, it lets us focus on the novel features of the
problem. For standard techniques to deal with the noncompact case see, for example,
\cite{donvar3}.

A construction of Markov jump processes was given in the \hyperref[intro]{Introduction},
and we
continue to use the notation introduced there. The jump intensity $q$ in
(\ref{generator}) is assumed to be continuous on $S$, and there exist
$0<K_{1}\leq K_{2}<\infty$ such that
%
\begin{equation}
K_{1}\leq q ( x ) \leq K_{2}. \label{qk}%
\end{equation}

Reversibility seems necessary to obtain an explicit formula for the rate
function, and we will make such an assumption. Recall that $\mathcal
{D}$ is
the domain of $\mathcal{L}$.

\begin{condition}
$\mathcal{L}$ is self-adjoint (or reversible) under $\pi$ in the following
sense: for any $f,g\in\mathcal{D}$
%
\begin{equation}
\int_{S} \bigl( \mathcal{L}f ( x ) \bigr) g ( x ) \pi ( dx
) =\int_{S} \bigl( \mathcal{L}g ( x ) \bigr) f ( x ) \pi (
dx ). \label{reversible1}%
\end{equation}
\end{condition}

An equivalent condition for (\ref{reversible1}) to hold is the
``detailed balance'' condition, that is, for
$\pi$-a.e. $x,y\in S$
%
\begin{equation}
q ( x ) \alpha ( x,dy ) \pi ( dx ) =q ( y ) \alpha ( y,dx ) \pi ( dy ).
\label{reversible2}%
\end{equation}
Note that (\ref{reversible2}) directly implies $\int_{S} ( \mathcal
{L}%
f ( x )  ) \pi ( dx ) =0$ for all $f\in
\mathcal{D}$.

To ensure ergodicity of $X ( t ) $, we need several conditions on
the transition function $\alpha$ in (\ref{transition}). Recall that
$\mathcal{P} ( S ) $ is the metric space of probability
measures on
$S$ equipped with L\'evy--Prohorov metric, which is compatible with the topology
of weak convergence.

\begin{condition}
\label{cond1}$\alpha$ satisfies the Feller property. That is, $\alpha
(
x,\cdot ) \dvtx S\longmapsto\mathcal{P} ( S ) $ is
continuous in
$x$.
\end{condition}

\begin{remark}
The Feller property and the compactness of $S$ guarantee $\alpha$ has an
invariant distribution (\cite{dupell4}, Proposition~8.3.4), which we
denote by
$\widetilde{\pi}$. The boundedness of $q$ enables us to define a probability
measure $\pi$ according to
%
\begin{equation}
\pi ( A ) \doteq\frac{\int_{A}({1}/{q ( x )
})\widetilde{\pi} ( dx ) }{\int_{S}({1}/{q ( x ))
}\widetilde{\pi} ( dx ) }. \label{pi_pi_tilde}%
\end{equation}
Since $\widetilde{\pi}$ is invariant under $\alpha$, that is,
$\widetilde{\pi
} ( \cdot ) =\int_{S}\alpha ( x,\cdot ) \widetilde{\pi
} ( dx ) $, we have
\[
\int_{S} \bigl( \mathcal{L}f ( x ) \bigr) \pi ( dx ) =
\frac{1}{\int_{S}({1}/{q ( x ) })\widetilde{\pi} (
dx ) }\int_{S}\int_{S} \bigl[
f ( y ) -f ( x ) \bigr] \alpha ( x,dy ) \widetilde{\pi} ( dx ) =0.
\]
By Echeverria's theorem (\cite{ethkur}, Theorem~4.9.17), $\pi$ is an invariant
distribution of $X ( t ) $.
\end{remark}

\begin{condition}
\label{cond2}$\alpha$ satisfies the following transitivity condition. There
exist positive integers $l_{0}$ and $n_{0}$ such that for all $x$ and
$\zeta$
in $S$
\[
\sum_{i=l_{0}}^{\infty}\frac{1}{2^{i}}
\alpha^{ ( i ) } ( x,dy ) \ll\sum_{j=n_{0}}^{\infty}
\frac{1}{2^{j}}\alpha^{ (
j ) } ( \zeta,dy ),
\]
where $\alpha^{ ( k ) }$ denotes the $k$-step transition probability.
\end{condition}

\begin{remark}
\label{rmk2}Under this condition, $\widetilde{\pi}$ is the unique invariant
distribution of~$\alpha$ (\cite{dupell4}, Lemma~8.6.2). Thus, $\pi$
defined by
(\ref{pi_pi_tilde}) is the unique probability distribution that satisfies
$\int_{S} ( \mathcal{L}f ( x )  ) \pi ( dx )
=0$, and hence by \cite{ethkur}, Theorem~4.9.17, is the unique invariant
distribution of~$X ( t ) $.
\end{remark}

\begin{condition}
\label{cond2'}There exists an integer $N$ and a positive real number
$c$ such
that
\[
\alpha^{(N)} ( x,\cdot ) \leq c\widetilde{\pi} ( \cdot )
\]
for all $x\in S$.
\end{condition}

\begin{remark}
This type of assumption is common in the large deviation analysis of empirical
measures. See, for example, \cite{ellwyn}, Hypothesis 1.1.
\end{remark}

\begin{condition}
\label{cond3}The support of $\pi$ is $S$.
\end{condition}

\begin{remark}
\label{isolate}This condition guarantees that any probability measure
$\eta
\in\mathcal{P} ( S ) $ can be approximated by measures that are
absolutely continuous with respect to $\pi$. Indeed, given $\delta>0$ let
$ \{ x_{j},N_{j},j=1,\ldots,J \} $ be such that $J<\infty$,
$x_{j}\in N_{j}\in\mathcal{B} ( S ) $, the $N_{j}$ are disjoint,
$\bigcup_{j=1}^{J}N_{j}=S$, $\pi(N_{j})>0$ and $\sup \{
d(x_{i},y)\dvtx y\in
N_{j} \} \leq\delta$ for $j=1,\ldots,J$ (this can be done by an open
covering argument). Given any $\eta\in\mathcal{P} ( S ) $ and
$A\in\mathcal{B} ( S ) $, let
\[
\bar{\eta}^{\delta}(A)=\sum_{j=1}^{J}
\frac{\pi(A\cap N_{j})}{\pi(N_{j})} 
\eta(N_{j}).
\]
Then $\bar{\eta}^{\delta}$ is absolutely continuous with respect to $\pi$.
Since $\bar{\eta}^{\delta}(N_{j})=\eta(N_{j})$ and $\sup \{
d(x_{i},y)\dvtx y\in N_{j} \} \leq\delta$, $\bar{\eta}^{\delta}%
\rightarrow\eta$ in the weak topology as $\delta\rightarrow0$.
\end{remark}

\begin{remark}
Condition \ref{cond3} excludes the existence of transient states.
Although one
can obtain an LDP for $X ( t ) $ that has transient states, one
would end up with a rate function that depends on the initial state.
\end{remark}

\section{A large deviation principle}
\label{state_ldp}

\subsection{Definition of the rate function}
\label{rate_defn}

In this subsection, we define the rate function $I$. In later sections, we
prove that $I$ thus defined is the correct form of the large deviation rate
function for the empirical measures of the Markov jump processes. All
conditions stated in Section~\ref{setup} will be assumed throughout the rest
of the paper. We wish to study the large deviation principle for the empirical
measures $\eta_{T}\in\mathcal{P} ( S ) $ defined by (\ref{eta}).
Under compactness of $S$ and Condition \ref{cond1}, $\eta_{T}$
converges in
distribution to an invariant distribution of $\mathcal{L}$. As pointed
out in
Remark~\ref{rmk2}, $\pi$ is the unique invariant distribution of
$\mathcal{L}%
$, and thus $\eta_{T}$ converges in distribution to $\pi$. Let $H$ be the
collection of all distributions that are absolutely continuous with
respect to
$\pi$, that is,
%
\begin{equation}
H\doteq \bigl\{ \eta\in\mathcal{P} ( S ) \dvtx\eta\ll\pi \bigr\}.
\label{H}%
\end{equation}
For $\eta\in H$, and assuming that the integral is well defined, consider
\[
-\int_{S}\theta^{1/2} ( x ) \mathcal{L} \bigl(
\theta^{1/2} ( x ) \bigr) \pi ( dx ),
\]
where $\theta=d\eta/d\pi$. This is a rewriting of $\Vert(-\mathcal{\bar
{L}%
)}^{1/2}g\Vert^{2}$ in (\ref{dvr}). By inserting the form of $\mathcal
{L}$ from
(\ref{generator}), we obtain the candidate rate function
%
\begin{equation}
I ( \eta ) =\int_{S}q ( x ) \eta ( dx ) -\int
_{S\times S}\theta^{1/2} ( x ) \theta^{1/2} ( y ) q (
x ) \alpha ( x,dy ) \pi ( dx ). \label{I1}%
\end{equation}

Note that by applying (\ref{reversible2}) and using the Cauchy--Schwarz
inequality, one can prove that $I$ defined by (\ref{I1}) is nonnegative.
Recall that $K_{2}$ is the upper bound of $q$ as in (\ref{qk}), and
thus $I$
is bounded above by $K_{2}$. In addition, it is straightforward to show that
$I$ is convex on $H$.

We want to extend the definition of $I$ to all measures in $\mathcal
{P} (
S ) $. As pointed out in Remark~\ref{isolate}, $H$ is dense in
$\mathcal{P} ( S ) $ under the topology of weak convergence. Hence,
we can extend the definition of $I$ to all of $\mathcal{P} ( S
) $
via lower semicontinuous regularization with respect to the topology of weak
convergence. Thus, if $\eta_{n}\rightarrow\eta$ weakly and $ \{ \eta
_{n} \} \in H$, $\liminf_{n\rightarrow\infty}I ( \eta_{n} )
\geq I ( \eta ) $, and equality holds for at least one such
sequence. This extension guarantees that the extended $I$ is convex, lower
semicontinuous and bounded above by $K_{2}$ on all of $\mathcal{P} (
S ) $. The compactness of $S$ and the lower semicontinuity of $I$ ensure
that $I$ has compact level sets. Being a nonnegative, lower semicontinuous
function with compact level sets, $I$ indeed is a valid large deviation
rate function.

We have finished the definition of the rate function $I$, and are now
ready to
state the large deviation principle.

\subsection{A large deviation principle}

Our main result is the following.

\begin{theorem}
\label{large dev} Let $X ( t ) $ be a Markov jump process
satisfying all the assumptions in Section~\ref{setup}. Let $I$ be
defined as
in Section~\ref{rate_defn}. Then the large deviation bounds (\ref{ldl}) and
(\ref{ldu}) hold.
\end{theorem}

To prove Theorem~\ref{large dev}, it suffices to show the equivalent Laplace
principle (\cite{dupell4}, Theorem~1.2.3). Specifically, we establish
that for
any bounded continuous function $F\dvtx\mathcal{P} ( S )
\rightarrow%
\mathbb{R}
$
%
\begin{equation}
\lim_{T\rightarrow\infty}-\frac{1}{T}\log E \bigl[ \exp \bigl\{ -TF (
\eta_{T} ) \bigr\} \bigr] =\inf_{\eta\in\mathcal{P} ( S )
} \bigl[ F ( \eta
) +I ( \eta ) \bigr]. \label{laplace}%
\end{equation}
By adding a constant to both sides of (\ref{laplace}), we can assume
$F\geq0$
and do that for the rest of the paper. The proof is based on a weak
convergence approach and is split into two parts: a Laplace upper bound
and a
Laplace lower bound.

Relative entropy plays a key role in the proof, and we hence state the
definition and a few important properties. Details can be found in
\cite{dupell4}.

\begin{definition}
\label{defn1}Let $ ( \mathcal{V},\mathcal{A} ) $ be a measurable
space. For $\vartheta\in\mathcal{P} ( \mathcal{V} ) $, the
\textit{relative entropy }$R ( \cdot\llVert \vartheta
)
$ is a mapping from $\mathcal{P} ( \mathcal{V} ) $ into the
extended real numbers. It is defined by
\[
R ( \gamma\llVert \vartheta ) \doteq\int_{\mathcal
{V}%
}
\biggl( \log\frac{d\gamma}{d\vartheta} \biggr) \,d\gamma
\]
when $\gamma\in$ $\mathcal{P} ( \mathcal{V} ) $ is absolutely
continuous with respect to $\vartheta$ and $\log \,d\gamma/d\vartheta$ is
integrable with respect to $\gamma$. Otherwise, we set $R ( \gamma
\llVert \vartheta  ) \doteq\infty$.
\end{definition}

If $\mathcal{V}$ is a Polish space and $\mathcal{A}$ the associated
$\sigma
$-algebra, then $R ( \cdot\llVert \cdot  ) $ is
nonnegative, jointly convex and jointly lower semicontinuous [with
respect to
the weak topology on $\mathcal{P} ( \mathcal{V} ) ^{2}$]. We state
the following two properties of relative entropy.

\begin{lemma}[(Variational formula)]
\label{var_for} Let $ ( \mathcal{V},\mathcal
{A}%
) $ be a measurable space, $k$ a bounded measurable function mapping
$\mathcal{V}$ into $%
\mathbb{R}
$, and $\vartheta$ a probability measure on~$\mathcal{V}$. The following
conclusions hold:
\begin{longlist}[(a)]
\item[(a)] We have the variational formula
%
\begin{equation}
-\log\int_{\mathcal{V}}e^{-k}\,d\vartheta=\inf
_{\gamma\in\mathcal{P} (
\mathcal{V} ) } \biggl\{ R ( \gamma\llVert \vartheta )
+\int_{\mathcal{V}}k\,d\gamma \biggr\}. \label{var_f}%
\end{equation}

\item[(b)] The infimum in (\ref{var_f}) is attained uniquely at $\gamma_{0}$ defined
by
\[
\frac{d\gamma_{0}}{d\vartheta} ( x ) \doteq e^{-k ( x )
}\Big/\int_{\mathcal{V}}e^{-k}\,d
\vartheta.
\]
\end{longlist}
\end{lemma}

\begin{theorem}[(Chain rule)]
\label{chain_rule} Let $\mathcal{X}$ and $\mathcal{Y}$ be Polish
spaces and $\beta$ and $\gamma$ probability measures on $\mathcal
{X\times Y}$.
We denote by $[\beta]_{1}$ and $[\gamma]_{1}$ the first marginals of
$\beta$
and $\gamma$ and by $\beta ( dy|x ) $ and $\gamma (
dy|x ) $ the stochastic kernels on $\mathcal{Y}$ given $\mathcal
{X}$ for
which we have the decompositions
\[
\beta ( dx\times dy ) =[\beta]_{1} ( dx ) \otimes \beta ( dy|dx )
\quad\mbox{and}\quad\gamma ( dx\times dy ) =[\gamma]_{1} ( dx ) \otimes\gamma (
dy|dx ).
\]
Then the function mapping $x\in\mathcal{X}\rightarrow R ( \beta (
\cdot|x ) \llVert \gamma ( \cdot|x )   ) $
is measurable and
\[
R ( \beta\llVert \gamma ) =R \bigl( [\beta ]_{1}
\llVert [\gamma]_{1} \bigr) +\int_{\mathcal{X}}R
\bigl( \beta ( \cdot|x ) \llVert \gamma ( \cdot|x )  \bigr) [
\beta]_{1} ( dx ).
\]
\end{theorem}

We devote the next two sections to proving the Laplace upper bound and the
Laplace lower bound, respectively.

\section{Proof of the Laplace upper bound}
\label{lpub}

In this section, we prove the Laplace upper bound part of (\ref{laplace}),
that is,
%
\begin{equation}
\liminf_{T\rightarrow\infty}-\frac{1}{T}\log E \bigl[ \exp \bigl\{ -TF
( \eta_{T} ) \bigr\} \bigr] \geq\inf_{\eta\in\mathcal{P} (
S ) } \bigl[ F
( \eta ) +I ( \eta ) \bigr]. \label{upper}%
\end{equation}

Recalling the construction of $X(t)$ in the \hyperref[intro]{Introduction}, we define a random
integer $R_{T}$ as the index when the total ``waiting
time'' first exceeds $T$, that is,
%
\begin{equation}
\sum_{i=1}^{R_{T}-1}s_{i}\leq T<\sum
_{i=1}^{R_{T}}s_{i}.
\label{R_T}%
\end{equation}
Then the empirical measure $\eta_{T}$ can be written as
%
\begin{eqnarray}
\label{eta_T}%
\eta_{T} ( \cdot ) & =&\frac{1}{T}\int
_{0}^{T}\delta_{X (
t ) } ( \cdot ) \,dt
\nonumber
\\
& =&\frac{1}{T} \Biggl[ \sum_{i=1}^{R_{T}-1}
\delta_{X_{i-1}} ( \cdot ) s_{i}+\delta_{X_{R_{T}-1}} ( \cdot )
\Biggl( T-\sum_{i=1}^{R_{T}-1}s_{i}
\Biggr) \Biggr]
\\
& =&\frac{1}{T} \Biggl[ \sum_{i=1}^{R_{T}-1}
\delta_{X_{i-1}} ( \cdot ) \frac{\tau_{i}}{q ( X_{i-1} ) }+\delta _{X_{R_{T}-1}%
} (
\cdot ) \Biggl( T-\sum_{i=1}^{R_{T}-1}
\frac{\tau_{i}}{q (
X_{i-1} ) } \Biggr) \Biggr].\nonumber
\end{eqnarray}
The proof of (\ref{upper}) will be partitioned into two cases:
$R_{T}/T>C$ and
$0\leq R_{T}/T\leq C$, where $C$ will be sent to $\infty$ after sending
$T\rightarrow\infty$.
\subsection{The case $R_{T}/T>C$}

Let $F\dvtx\mathcal{P}(S)\rightarrow\mathbb{R}$ be nonnegative and continuous.
Then since $F\geq0$,
\begin{eqnarray*}
-\frac{1}{T}\log E \bigl[ 1_{ \{  ( C,\infty )  \}
}(R_{T}/T)e^{-TF(\eta_{T})}
\bigr] & \geq&-\frac{1}{T}\log P \Biggl\{ \sum_{i=1}^{ \lfloor TC \rfloor+1}s_{i}
\leq T \Biggr\}
\\
& =&-\frac{1}{T}\log P \Biggl\{ \sum_{i=1}^{ \lfloor TC \rfloor
+1}
\frac{\tau_{i}}{q ( X_{i-1} ) }\leq T \Biggr\}
\\
& \geq&-\frac{1}{T}\log P \Biggl\{ \sum_{i=1}^{ \lfloor TC \rfloor
+1}
\tau_{i}\leq K_{2}T \Biggr\}.
\end{eqnarray*}
Using Chebyshev's inequality, for any $\alpha\in ( 0,\infty ) $,
\begin{eqnarray*}
P \Biggl\{ \sum_{i=1}^{ \lfloor TC \rfloor+1}
\tau_{i}\leq K_{2}T \Biggr\} & =&P \bigl\{ e^{-\alpha\sum_{i=1}^{ \lfloor
TC \rfloor+1}\tau_{i}}
\geq e^{-\alpha K_{2}T} \bigr\}
\\
& \leq& e^{\alpha K_{2}T}E \bigl[ e^{-\alpha\sum_{i=1}^{ \lfloor
TC \rfloor+1}\tau_{i}} \bigr]
\\
& =&e^{\alpha K_{2}T}e^{ (  \lfloor TC \rfloor+1 )
\log({1}/{1(+\alpha)})}.
\end{eqnarray*}
For the last equality, we have used that if $\tau$ is exponentially distributed
with mean $1$ then $Ee^{a\tau}=1/ ( 1-a ) $ for any $a\in (
-\infty,1 ) $. Combining the last two inequalities,
\begin{eqnarray*}
&&\liminf_{T\rightarrow\infty}  -\frac{1}{T}\log E \bigl[
1_{ \{  (
C,\infty )  \} }(R_{T}/T)\cdot\exp\bigl\{-TF(\eta_{T})\bigr
\} \bigr]
\\
&&\qquad \geq\sup_{\alpha\in ( 0,\infty ) } \bigl[ -K_{2}\alpha +C\log ( 1+
\alpha ) \bigr]
\\
&&\qquad =-C+C\log C+K_{2}-C\log K_{2}.
\end{eqnarray*}
Note that $-C+C\log C+K_{2}-C\log K_{2}\rightarrow\infty$ as
$C\rightarrow
\infty$.

\subsection{The case \texorpdfstring{$0\leq R_{T}/T\leq C$}{$0<=R_{T}/T<=C$}}

\subsubsection{A stochastic control representation}

In this case, we adapt a standard weak convergence argument; see \cite{dupell4}
for details. Specifically, we first\vadjust{\goodbreak} establish a stochastic control
representation for the left-hand side of (\ref{laplace}) and then
obtain a
lower bound for the limit as $T\rightarrow\infty$. In the
representation, all
distributions can be perturbed from their original form, but such a
perturbation pays a relative entropy cost. We distinguish the new
distributions and random variables by an overbar. In the following, the barred
quantities are constructed analogously to their unbarred counterparts. Hence,
$\bar{\tau}_{i}$ and $\bar{X}_{i}$ are chosen recursively according to
stochastic kernels $\bar{\sigma}_{i} ( \cdot ) $ and $\bar
{\alpha
}_{i} ( \cdot ) $, that is, $\bar{\sigma}_{i} ( \cdot
) $
and $\bar{\alpha}_{i} ( \cdot ) $ are conditional distributions
that can depend on the whole past. Specifically, $\bar{\sigma}_{i} (
\cdot ) $ depends on $ \{ \bar{X}_{0},\bar{\tau}_{1},\bar{X}%
_{1},\bar{\tau}_{2},\ldots,\bar{X}_{i-1} \} $ and $\bar{\alpha}%
_{i} ( \cdot ) $ depends on $ \{ \bar{X}_{0},\bar{\tau}%
_{1},\bar{X}_{1},\bar{\tau}_{2},\ldots,\bar{X}_{i-1},\bar{\tau}_{i}
\} $;
$\bar{s}_{i}$ is defined by (\ref{soj}) using $\bar{X}_{i}$ and $\bar
{\tau
}_{i}$; $\bar{R}_{T}$ is defined by (\ref{R_T}) using $\bar{s}_{i}$; and
$\bar{\eta}_{T}$ is defined by (\ref{eta_T}) using $\bar{X}_{i}$, $\bar
{\tau
}_{i}$ and $\bar{R}_{T}$. It will be sufficient to consider any deterministic
sequence $ \{ r_{T} \} $ such that $0\leq r_{T}/T\leq C$, and
$r_{T}/T\rightarrow A$ for some $A\in [ 0,C ] $ as
$T\rightarrow
\infty$. We restrict consideration to controlled processes such that
$\bar
{R}_{T}=r_{T}$ by placing an infinite cost penalty on controls which
lead to
any other outcome with positive probability. Let $\mathbf{1} (
A ) $ denote the indicator function of a set $A$, and recall that our
convention is $0\cdot\infty=0$. By applying \cite{dupell4}, Proposition~4.5.1,
and Theorem~\ref{chain_rule}, the following is valid:
%
\begin{eqnarray}\label{rep_gen}
&& -\frac{1}{T}\log E \bigl[ \exp \bigl\{ -TF(\eta_{T})-T\cdot
\infty \cdot\mathbf{1} ( r_{T}\neq R_{T} ) \bigr\} \bigr]
\nonumber
\\[-1pt]
\\[-15pt]
\nonumber
&&\qquad =-\frac{1}{T}\log E \Biggl[ \exp \Biggl\{ -TF(\eta_{T})-T\cdot
\infty \cdot\mathbf{1} \Biggl( \Biggl\{ \sum_{i=1}^{r_{T}-1}s_{i}
\leq T<\sum_{i=1}^{r_{T}}s_{i} \Biggr
\} ^{c} \Biggr) \Biggr\} \Biggr]
\\
\label{rep_gen1}&&\qquad =\inf E \Biggl[ F ( \bar{\eta}_{T} ) +\infty\cdot\mathbf{1}
\Biggl( \Biggl\{ \sum_{i=1}^{r_{T}-1}
\bar{s}_{i}\leq T<\sum_{i=1}^{r_{T}}%
\bar{s}_{i} \Biggr\} ^{c} \Biggr)
\nonumber
\\[-8pt]
\\[-8pt]
\nonumber
&&\hspace*{99pt}{} +\frac{1}{T}\sum
_{i=1}^{r_{T}} \bigl[ R ( \bar{
\alpha}_{i-1}\llVert \alpha ) +R ( \bar{
\sigma}_{i}\llVert \sigma ) \bigr] \Biggr],
\end{eqnarray}
where the infimum is taken over all control measures $ \{ \bar
{\alpha
}_{i},\bar{\sigma}_{i} \} $. Since in Section~\ref{lplb} we will
prove a
similar but more involved representation formula, Lemma~\ref{rep_ran},
we omit
the proof of this representation. Due to the restriction $\bar{R}_{T}=r_{T}$,
one can write $\bar{\eta}_{T}$ as
%
\begin{equation}
\bar{\eta}_{T} ( \cdot ) =\frac{1}{T} \Biggl[ \sum
_{i=1}^{r_{T}%
-1}\delta_{\bar{X}_{i-1}} ( \cdot )
\frac{\bar{\tau
}_{i}}{q (
\bar{X}_{i-1} ) }+\delta_{\bar{X}_{r_{T}-1}} ( \cdot ) \Biggl( T-\sum
_{i=1}^{r_{T}-1}\frac{\bar{\tau}_{i}}{q ( \bar{X}%
_{i-1} ) } \Biggr) \Biggr].\hspace*{-15pt}
\label{eta_bar}%
\end{equation}

In the following proof, we repeatedly extract further subsequences of
$T$. To
keep the notation concise, we abuse notation and use $T$ to denote all
subsequences. Note also that in proving a lower bound for (\ref
{rep_gen}) it
suffices to consider a subsequence of $T$ such that
%
\begin{equation}
\sup_{T}-\frac{1}{T}\log E \bigl[ \exp \bigl\{ -TF(
\eta_{T})-T\cdot\infty \cdot\mathbf{1} ( r_{T}\neq
R_{T} ) \bigr\} \bigr] <\infty. \label{finiteness}%
\end{equation}
We assume this condition for the rest of this subsection.

The relative entropy cost in (\ref{rep_gen1}) includes two parts,
$RE_{T}%
^{1}\doteq\break \frac{1}{T}\sum_{i=1}^{r_{T}}R ( \bar{\alpha}_{i-1}\llVert
\alpha  ) $ and $RE_{T}^{2}\doteq\frac{1}{T}\sum_{i=1}^{r_{T}%
}R ( \bar{\sigma}_{i}\llVert \sigma  ) $.\vspace*{1pt} We will prove
that for any sequence of controls $ \{ \bar{\alpha}_{i},\bar{\sigma}
_{i} \} $ in (\ref{rep_gen1})
%
\begin{equation}
\liminf_{T\rightarrow\infty}E \bigl[ F ( \bar{\eta}_{T} )
+RE_{T}^{1}+RE_{T}^{2} \bigr] \geq\inf
_{\eta\in\mathcal{P} ( S )
} \bigl[ F ( \eta ) +I ( \eta ) \bigr].
\label{lower1}%
\end{equation}
Toward this end, it is enough to show that along any subsequence of $T$ such
that $r_{T}/T\rightarrow A$, we can extract a further subsequence along which
(\ref{lower1}) holds. In addition, it suffices to consider only
functions $F$
that besides being nonnegative, are also lower semicontinuous and
convex. This
restriction is valid since $I$ is convex and lower semicontinuous, and follows
a standard argument in the large deviation literature. The interested reader
can find the details in~\cite{liu2}.

In light of (\ref{rep_gen1}) and (\ref{finiteness}), we assume without
loss of
generality
%
\begin{equation}
\sup_{T}E \bigl[ F ( \bar{\eta}_{T} )
+RE_{T}^{1}+RE_{T}%
^{2} \bigr]
<\infty. \label{uni_finite}%
\end{equation}
Since the proof of (\ref{lower1}) is lengthy, we analyze each term on
the left-hand side of~(\ref{lower1}) separately in the following subsections.

\subsubsection{The term $RE_{T}^{1}$}
\label{term1}

The cost $RE_{T}^{1}$ comes from distorting the dynamics of the embedded
Markov chain, and indeed the analysis gives a very similar conclusion
to that
of an ordinary Markov chain (\cite{dupell4}, Chapter~8). For any probability
measure $\nu$ on $S\times S$, we will use notation $ [ \nu ] _{1}$
and $ [ \nu ] _{2}$ to denote the first and second marginals of
$\nu$. We have the following result for $RE_{T}^{1}$.

\begin{lemma}
\label{re1} Consider any sequence of controls $ \{ \bar{\alpha}_{i}%
,\bar{\sigma}_{i} \} $ in (\ref{rep_gen1}) such that~(\ref{uni_finite})
holds. Along any subsequence of $T$ satisfying $r_{T}/T\rightarrow A$, define
a sequence of random probability measures on $S\times S$ via
\[
\mu_{T} ( dx,dy ) \doteq\frac{1}{r_{T}}\sum
_{i=1}^{r_{T}}%
\delta_{\bar{X}_{i-1}} ( dx )
\bar{\alpha}_{i-1} ( dy ) .
\]
Then one can extract a further subsequence such that $E\mu_{T}$
converges in
distribution to a probability measure $\tilde{\mu}$ on $S\times S$, and
\[
\liminf_{T\rightarrow\infty}E \bigl[ RE_{T}^{1} \bigr]
\geq AR \bigl( \tilde{\mu}\llVert [ \tilde{\mu} ] _{1}\otimes\alpha
 \bigr).
\]
Furthermore, if $A>0$ then $\tilde{\mu}$ satisfies
%
\begin{equation}
[ \tilde{\mu} ] _{1}= [ \tilde{\mu} ] _{2}.
\label{const1}%
\end{equation}
\end{lemma}

\begin{pf}
By the chain rule (Theorem~\ref{chain_rule}) and the joint convexity of
relative entropy,
\begin{eqnarray*}
E \bigl[ RE_{T}^{1} \bigr] & =&E \Biggl[ \frac{1}{T}
\sum_{i=1}^{r_{T}%
}R ( \bar{
\alpha}_{i-1}\llVert \alpha ) \Biggr]
\\
& =&E \Biggl[ \frac{1}{T}\sum_{i=1}^{r_{T}}R
( \delta_{\bar{X}_{i-1}%
}\otimes\bar{\alpha}_{i-1}\llVert
\delta_{\bar{X}_{i-1}}\otimes \alpha ) \Biggr]
\\
& \geq& E \biggl[ \frac{r_{T}}{T}R \bigl( \mu_{T}\llVert [ \mu
_{T} ] _{1}\otimes\alpha \bigr) \biggr]
\\
& \geq&\frac{r_{T}}{T}R \bigl( E\mu_{T}\llVert [ E\mu_{T}
] _{1}\otimes\alpha \bigr).
\end{eqnarray*}
Since $S\times S$ is compact, for any subsequence of $T$ there exists a
further subsequence along which $E\mu_{T}$ converges weakly to a probability
measure $\tilde{\mu}$. Under the Feller property of $\alpha$ (Condition
\ref{cond1}), $ [ E\mu_{T} ] _{1}\otimes\alpha$ converges
weakly to
$ [ \tilde{\mu} ] _{1}\otimes\alpha$. The lower semicontinuity of
relative entropy then implies
\[
\liminf_{T\rightarrow\infty}E \bigl[ RE_{T}^{1} \bigr]
\geq\liminf_{T\rightarrow\infty}\frac{r_{T}}{T}R \bigl( E\mu_{T}
\llVert [ E\mu_{T} ] _{1}\otimes\alpha \bigr) \geq AR
\bigl( \tilde {\mu }\llVert [ \tilde{\mu} ] _{1}\otimes\alpha
\bigr).
\]

This finishes the first part of Lemma~\ref{re1}. For the second part, we
employ a standard martingale argument. Let $\mathcal{F}_{i}$ be the
$\sigma
$-algebra generated by the random variables $ \{  ( \bar{X}%
_{0},\ldots,\bar{X}_{i} ), ( \bar{\tau}_{1},\ldots,\bar{\tau}%
_{i} )  \} $. Thus, $\mathcal{F}_{i}$ is a sequence of increasing
$\sigma$-algebra's and, since $\bar{\alpha}_{i}$ selects the conditional
distribution of $\bar{X}_{i}$, for any bounded continuous function $f$
on $S$
\[
E \biggl[  \biggl( f ( \bar{X}_{i} ) -\int
_{S}f ( y ) \bar{\alpha}_{i} ( dy ) \biggr)
\Big\rrvert \mathcal{F}_{i-1} \biggr] =0.
\]
Hence, for integers $0\leq i<k\leq r_{T}-1$
\[
E \biggl[ \biggl( f ( \bar{X}_{i} ) -\int_{S}f (
y ) \bar{\alpha}_{i} ( dy ) \biggr) \biggl( f ( \bar{X}%
_{k} ) -\int_{S}f ( y ) \bar{
\alpha}_{k} ( dy ) \biggr) \biggr] =0,
\]
and thus for any bounded continuous function $f$ on $S$
\begin{eqnarray*}
&& E \biggl[ \int_{S\times S}f ( x ) \mu_{T} ( dx,dy ) -
\int_{S\times S}f ( y ) \mu_{T} ( dx,dy ) \biggr]
^{2}
\\
&&\qquad =E \Biggl[ \frac{1}{r_{T}}\sum_{i=1}^{r_{T}}f
( \bar{X}_{i-1} ) -\frac{1}{r_{T}}\sum_{i=1}^{r_{T}}
\int_{S}f ( y ) \bar{\alpha }_{i-1} ( dy ) \Biggr]
^{2}
\\
&&\qquad \leq\frac{1}{r_{T}^{2}}\sum_{i=1}^{r_{T}}E
\biggl[ f ( \bar{X}%
_{i-1} ) -\int_{S}f (
y ) \bar{\alpha}_{i-1} ( dy ) \biggr] ^{2}
\\
&&\qquad \leq\frac{4}{r_{T}}\llVert f\rrVert _{\infty}^{2}.
\end{eqnarray*}
Since $0<A=\lim_{T\rightarrow\infty}r_{T}/T$, we have $r_{T}/T\geq A/2$ for
all $T$ large enough. Using Chebyshev's inequality and the last
display, we
conclude that $ [ \mu_{T} ] _{1}- [ \mu_{T} ] _{2}$
converges to $0$ in probability as $T\rightarrow\infty$ and, therefore,
$ [ \tilde{\mu} ] _{1}= [ \tilde{\mu} ] _{2}$ with
probability~$1$. This concludes the second part of Lemma~\ref{re1}.
\end{pf}

\subsubsection{The term $RE_{T}^{2}$}
\label{term2}

We now turn to the second cost $RE_{T}^{2}$. This cost comes from distorting
the exponential sojourn times. We introduce a function $\ell$ which is closely
related to the relative entropy of exponential distributions: $\ell (
x ) \doteq x\log x-x+1$ for any $x\geq0$.

\begin{lemma}
\label{re2}Given any sequence of controls $ \{ \bar{\alpha}_{i}%
,\bar{\sigma}_{i} \} $, fix a subsequence of $T$ for which the
conclusions in Lemma~\ref{re1} holds. Then we can extract a further
subsequence along which
\[
\liminf_{T\rightarrow\infty}E \bigl[ RE_{T}^{2} \bigr]
\geq\int_{S\times%
\mathbb{R}
_{+}}\ell ( u ) \tilde{\xi} ( dx,du ).
\]
Here, $\tilde{\xi}$ is a finite measure on $S\times%
\mathbb{R}
_{+}$ and is related to $\tilde{\mu}$ in Lemma~\ref{re1} by
%
\begin{equation}
\int_{%
\mathbb{R}
_{+}}u\tilde{\xi} ( dx,du ) =A [ \tilde{\mu} ]
_{1} ( dx ). \label{const2}%
\end{equation}
\end{lemma}

Before proving this lemma, we define $g$: $%
\mathbb{R}
_{+}\rightarrow%
\mathbb{R}
$ by $g ( b ) \doteq-\log b+b-1$. The functions $g$ and $\ell$ are
related by
\[
g ( x ) =x\ell ( 1/x ),
\]
and $g$ has the following property.

\begin{lemma}
\label{exp_re}Let $\sigma$ be an exponential distribution with mean
$1$. Then
%
\begin{equation}
\inf \biggl\{ R ( \gamma\llVert \sigma ) \dvtx\int
_{%
\mathbb{R}
_{+}}u\gamma ( du ) =b \biggr\} =g ( b ).
\label{exp_re_eq}%
\end{equation}
\end{lemma}

\begin{pf}
Let $\sigma_{b}$ be the exponential distribution with mean $b$, that is,
\[
\sigma_{b} ( du ) =\frac{1}{b}e^{-{u}/{b}}\,du.
\]
Then $\frac{d\sigma_{b}}{d\sigma} ( u ) =\frac{1}{b}e^{ (
1-{1}/{b} ) u}$ for $u>0$. Picking any $\gamma$ such that
$R (
\gamma\llVert \sigma  ) <\infty$ and $\int_{%
\mathbb{R}
_{+}}u\gamma ( du ) =b$,
\begin{eqnarray*}
R ( \gamma\llVert \sigma ) & =&\int_{%
\mathbb{R}
_{+}}
\log \biggl( \frac{d\gamma}{d\sigma} \biggr) \gamma ( du )
\\
& =&\int_{%
\mathbb{R}
_{+}}\log \biggl( \frac{d\gamma}{d\sigma_{b}} \biggr) \gamma
( du ) +\int_{%
\mathbb{R}
_{+}}\log \biggl( \frac{d\sigma_{b}}{d\sigma} \biggr)
\gamma ( du )
\\
& =&R ( \gamma\llVert \sigma_{b} ) +\int
_{%
\mathbb{R}
_{+}} \biggl[ -\log b+ \biggl( 1-\frac{1}{b} \biggr) u
\biggr] \gamma ( du )
\\
& =&R ( \gamma\llVert \sigma_{b} ) +g ( b )
\\
& \geq& g ( b )
\end{eqnarray*}
and the infimum in (\ref{exp_re_eq}) is achieved when $R ( \gamma
\llVert \sigma_{b}  ) =0$, that is, $\gamma=\sigma_{b}$.
\end{pf}

\begin{pf*}{Proof of Lemma~\ref{re2}}
Lemma~\ref{exp_re} guarantees that
%
\begin{equation}
RE_{T}^{2}\geq\frac{1}{T}\sum
_{i=1}^{r_{T}}g \biggl( \int u\bar{\sigma}%
_{i} ( du ) \biggr). \label{re2_g}%
\end{equation}
Recall the definition of $\mathcal{F}_{i}$ as the $\sigma$-algebra generated
by the controlled process up to time $i$. Since $\bar{\sigma}_{i}$
selects the
conditional distribution of $\bar{\tau}_{i}$,
\[
E [ \bar{\tau}_{i}|\mathcal{F}_{i-1} ] =\int u\bar{
\sigma}%
_{i} ( du ).
\]
Define $\bar{m}_{i}\doteq\int u\bar{\sigma}_{i} ( du ) $, for
$i=1,\ldots,r_{T}-1$. The definition of $\bar{m}_{r_{T}}$ requires more work.
Recalling the definition of $\bar{R}_{T}$ by the equation analogous to
(\ref{R_T}) and the restriction that $\bar{R}_{T}=r_{T}$,
\[
T-\sum_{i=1}^{r_{T}-1}\frac{\bar{\tau}_{i}}{q ( \bar{X}_{i-1} )
}\leq
\frac{\bar{\tau}_{r_{T}}}{q ( \bar{X}_{r_{T}-1} ) }.
\]
Multiplying both sides by $q ( \bar{X}_{r_{T}-1} ) $ and taking
expectation conditioned on $\mathcal{F}_{r_{T}-1}$,
\[
q ( \bar{X}_{r_{T}-1} ) \Biggl( T-\sum_{i=1}^{r_{T}-1}
\frac
{\bar{\tau}_{i}}{q ( \bar{X}_{i-1} ) } \Biggr) \leq E [ \bar{\tau}_{r_{T}}|
\mathcal{F}_{r_{T}-1} ] =\int u\bar{\sigma }_{r_{T}%
} ( du ).
\]
Define
%
\begin{equation}
\bar{\Delta}_{T}\doteq q ( \bar{X}_{r_{T}-1} ) \Biggl( T-\sum
_{i=1}^{r_{T}-1}\frac{\bar{\tau}_{i}}{q ( \bar{X}_{i-1} )
} \Biggr) ,
\label{deltat}%
\end{equation}
and define $\bar{m}_{r_{T}}$ by
%
\begin{equation}
\bar{m}_{r_{T}}\doteq\cases{
\displaystyle\int u\bar{\sigma}_{r_{T}} ( du ), &\quad $\mbox{if }\displaystyle\bar{
\Delta}_{T}%
\leq\int u\bar{\sigma}_{r_{T}} ( du ) <1,$
\vspace*{2pt}\cr
1, & \quad$\mbox{if }\displaystyle\bar{\Delta}_{T}\leq1\leq\int u\bar{
\sigma}_{r_{T}} ( du ),$
\vspace*{2pt}\cr
\bar{\Delta}_{T}, &\quad $\mbox{if }\displaystyle 1<\bar{\Delta}_{T}\leq\int u
\bar{\sigma }_{r_{T}%
} ( du ),$}
\label{mrt}%
\end{equation}
that is, $\bar{m}_{r_{T}}$ is the median of the triplet $ ( \bar
{\Delta}%
_{T},\int u\bar{\sigma}_{r_{T}} ( du ),1 ) $. Since $g$ is
increasing on $ ( 1,\infty ) $, we have $g ( \int u\bar
{\sigma}_{r_{T}} ( du )  ) \geq g ( \bar{m}_{r_{T}%
} ) $ in all three cases. Thus, by (\ref{re2_g}),
%
\begin{equation}
RE_{T}^{2}\geq\frac{1}{T}\sum
_{i=1}^{r_{T}}g \biggl( \int u\bar{\sigma}%
_{i} ( du ) \biggr) \geq\frac{1}{T}\sum
_{i=1}^{r_{T}}g ( \bar{m}_{i} ).
\label{gbd}%
\end{equation}

Next, consider the measure on $S\times%
\mathbb{R}
_{+}$ defined by
%
\begin{equation}
\xi_{T} ( dx,du ) \doteq\frac{1}{T}\sum
_{i=1}^{r_{T}}\delta _{\bar{X}_{i-1}} ( dx )
\delta_{ ( \bar{m}_{i} ) ^{-1}
} ( du ) \bar{m}_{i}. \label{xit}%
\end{equation}
The total mass of $E\xi_{T}$ is
\[
E\xi_{T} ( S\times%
\mathbb{R}
_{+} ) =\frac{1}{T}\sum
_{i=1}^{r_{T}}E [ \bar{m}_{i} ].
\]
According to (\ref{uni_finite}) and the assumption that $F\geq0$, we have
%
\begin{equation}
\sup_{T}E \bigl[ RE_{T}^{2} \bigr] <
\infty. \label{re_bd}%
\end{equation}
By (\ref{gbd}) $\sup_{T}E [ \sum_{i=1}^{r_{T}}g ( \bar{m}%
_{i} ) /T ] <\infty$. We also have by a straightforward
calculation that $x\leq\max \{ 50,10g ( x ) /9 \} $.
Using this and the fact that $r_{T}/T\leq C$, we have $\sup_{T}E [
\sum_{i=1}^{r_{T}}\bar{m}_{i}/T ] <\infty$, that is, the total
mass of
$E\xi_{T}$ has a bound uniform in $T$. Thus, when viewed as a sequence of
measures on the compact space $S\times [ 0,\infty ] $, $E\xi_{T}$
is tight due to the uniform boundedness of the total mass. We denote
the weak
limit by $\tilde{\xi}$, which is a finite measure. Since the function
$\ell$
is nonnegative and continuous,
%
\begin{eqnarray}\label{ll}
\liminf_{T\rightarrow\infty}E \bigl[ RE_{T}^{2} \bigr] &
\geq &\liminf_{T\rightarrow\infty}E \Biggl[ \frac{1}{T}\sum
_{i=1}^{r_{T}}g ( \bar{m}_{i} ) \Biggr]
\nonumber
\\
& =&\liminf_{T\rightarrow\infty}E \biggl[ \int_{S\times%
\mathbb{R}
_{+}}\ell (
u ) \xi_{T} ( dx,du ) \biggr]
\\
& =&\liminf_{T\rightarrow\infty}\int_{S\times%
\mathbb{R}
_{+}}\ell ( u ) E
\xi_{T} ( dx,du )
\nonumber\\
& \geq&\int_{S\times%
\mathbb{R}
_{+}}\ell ( u ) \tilde{\xi} ( dx,du ).
\nonumber
\end{eqnarray}

We next explore the relation between $\tilde{\xi}$ and $\tilde{\mu}$.
In order
to establish (\ref{const2}), it suffices to show that for any bounded
continuous function $f$ on $S$
\[
\int_{S\times%
\mathbb{R}
_{+}}uf ( x ) \tilde{\xi} ( dx,du ) =A\int
_{S}f ( x ) [ \tilde{\mu} ] _{1} ( dx ).
\]
By the definitions of $\xi_{T}$ and $\mu_{T,}$
%
\begin{equation}
\int_{S\times%
\mathbb{R}
_{+}}uf ( x ) E\xi_{T} ( dx,du ) =
\frac{r_{T}}{T}%
\int_{S}f ( x ) [ E
\mu_{T} ] _{1} ( dx ). \label{exi_emu}%
\end{equation}
Then (\ref{re_bd}) and (\ref{ll}) imply there is a uniform upper bound on
\[
\int_{%
\mathbb{R}
_{+}}\ell ( u ) \int_{S}f ( x ) E
\xi_{T} ( dx,du ).
\]
If we consider $\int_{S}f ( x ) E\xi_{T} ( dx,du ) $ as
a sequence of measures on $%
\mathbb{R}
_{+}$ with bounded total mass, then $\int_{S}f ( x ) E\xi
_{T} ( dx,du ) $ converges weakly to $\int_{S}f ( x )\*
\tilde{\xi} ( dx,  du ) $. Since $\ell$ is superlinear, \cite
{dupell4}, Theorem
A.3.19, implies that
\[
\lim_{T\rightarrow\infty}\int_{S\times%
\mathbb{R}
_{+}}uf ( x ) E
\xi_{T} ( dx,du ) =\int_{S\times%
\mathbb{R}
_{+}}uf ( x ) \tilde{\xi} (
dx,du ).
\]
Using
\[
\lim_{T\rightarrow\infty}\frac{r_{T}}{T}\int_{S}f ( x
) [ E\mu_{T} ] _{1} ( dx ) =A\int_{S}f (
x ) [ \tilde{\mu} ] _{1} ( dx )
\]
and (\ref{exi_emu}) we arrive at (\ref{const2}).
\end{pf*}

\subsubsection{The term \texorpdfstring{$E\bar{\eta}_{T}$}{$E bar{eta}_{T}$}}
\label{term3}

\begin{lemma}
\label{eeta} Given any sequence of controls $ \{ \bar{\alpha}%
_{i},\bar{\sigma}_{i} \} $, fix a subsequence of $T$ for which the
conclusions in Lemma~\ref{re2} hold. Then we can extract a further subsequence
along which
\[
\liminf_{T\rightarrow\infty}E \bigl[ F ( \bar{\eta}_{T} ) \bigr]
\geq F ( \tilde{\eta} )
\]
for some probability measure $\tilde{\eta}$ on $S$, which is related to
$\tilde{\xi}$ in Lemma~\ref{re2} by
%
\begin{equation}
q ( x ) \tilde{\eta} ( dx ) =[\tilde{\xi}]_{1} ( dx ).
\label{const3}%
\end{equation}
\end{lemma}

\begin{pf}
As a sequence of probability measures on the compact space $S$, we can always
extract a subsequence of $T$ such that $E\bar{\eta}_{T}$ converges
weakly to a
probability measure on $S$ which we denote by $\tilde{\eta}$. The convexity
and lower semicontinuity of $F$ imply that
\[
\liminf_{T\rightarrow\infty}E \bigl[ F ( \bar{\eta}_{T} ) \bigr]
\geq\liminf_{T\rightarrow\infty}F ( E\bar{\eta}_{T} ) \geq F (
\tilde{\eta} ).
\]
By the definitions of $\bar{\eta}_{T}$ in (\ref{eta_bar}) and $\bar
{\Delta
}_{T}$ in (\ref{deltat}),
\begin{eqnarray*}
& &q ( x ) E\bar{\eta}_{T} ( dx )
\\
& &\qquad=\frac{q ( x ) }{T}E \Biggl[ \sum_{i=1}^{r_{T}-1}
\delta_{\bar
{X}_{i-1}} ( dx ) \frac{\bar{\tau}_{i}}{q ( \bar{X}%
_{i-1} ) }+\delta_{\bar{X}_{r_{T}-1}} ( dx )
\Biggl( T-\sum_{i=1}^{r_{T}-1}\frac{\bar{\tau}_{i}}{q ( \bar{X}_{i-1} )
}
\Biggr) \Biggr]
\\
&&\qquad =\frac{1}{T}E \Biggl[ \sum_{i=1}^{r_{T}-1}
\delta_{\bar{X}_{i-1}} ( dx ) \bar{\tau}_{i}+\delta_{\bar{X}_{r_{T}-1}} (
dx ) \bar{\Delta}_{T} \Biggr]
\\
&&\qquad =\frac{1}{T} \Biggl( \sum_{i=1}^{r_{T}-1}E
\bigl[ E \bigl[ \delta_{\bar
{X}_{i-1}} ( dx ) \bar{\tau}_{i}|
\mathcal{F}_{i-1} \bigr] \bigr] +E \bigl[ \delta_{\bar{X}_{r_{T}-1}} ( dx )
\bar{\Delta}%
_{T} \bigr] \Biggr)
\\
&&\qquad =\frac{1}{T} \Biggl( \sum_{i=1}^{r_{T}-1}E
\bigl[ \delta_{\bar{X}_{i-1}%
} ( dx ) \bar{m}_{i} \bigr] +E \bigl[
\delta_{\bar{X}_{r_{T}-1}%
} ( dx ) \bar{\Delta}_{T} \bigr] \Biggr).
\end{eqnarray*}
Recalling the definition of $\xi_{T}$ in (\ref{xit}), we have
\[
[ E\xi_{T} ] _{1} ( dx ) =\frac{1}{T}\sum
_{i=1}^{r_{T}%
}E \bigl[ \delta_{\bar{X}_{i-1}} ( dx )
\bar{m}_{i} \bigr].
\]
This implies the total variation bound
\[
\bigl\llVert q ( x ) E\bar{\eta}_{T} ( dx ) - [ E\xi_{T} ]
_{1} ( dx ) \bigr\rrVert _{\mathrm{TV}}\leq\frac
{1}{T}E\llvert
\bar{\Delta}_{T}-\bar{m}_{r_{T}}\rrvert .
\]
Recalling the definition of $\bar{m}_{r_{T}}$ in (\ref{mrt}), we
conclude that
\[
\bigl\llVert q ( x ) E\bar{\eta}_{T} ( dx ) - [ E\xi_{T} ]
_{1} ( dx ) \bigr\rrVert _{\mathrm{TV}}\leq\frac
{1}{T}.
\]
By taking limits, we arrive at (\ref{const3}).
\end{pf}

Lemmas \ref{re1}, \ref{re2} and~\ref{eeta} together imply
for a
sequence of controls $ \{ \bar{\alpha}_{i},\bar{\sigma}_{i} \} $
satisfying (\ref{rep_gen1}), along any subsequence of $T$ such that
$r_{T}/T\rightarrow A$, we can extract a further subsequence along which
%
\begin{eqnarray}\label{lower2}%
&&\liminf_{T\rightarrow\infty}E \bigl[ F ( \bar{\eta}_{T} )
+RE_{T}^{1}+RE_{T}^{2} \bigr]
\nonumber
\\[-8pt]
\\[-8pt]
\nonumber
&&\qquad\geq F (
\tilde{\eta} ) +AR \bigl( \tilde{\mu}\llVert [ \tilde{\mu} ] _{1}\otimes
\alpha \bigr) +\int_{S\times%
\mathbb{R}
_{+}}\ell ( u ) \tilde{\xi} (
dx,du ),
\end{eqnarray}
where $\tilde{\eta}$, $\tilde{\mu}$ and $\tilde{\xi}$ satisfy the constraints
(\ref{const2}), (\ref{const3}) and (\ref{const1}) if $A>0$.

Recall that our goal is to prove (\ref{lower1}). Hence, we need to establish
the relationship between the right-hand side of (\ref{lower2}) and the rate
function $I$ defined in Section~\ref{rate_defn}.

\subsubsection{Properties of the rate function $I$}

We prove the following lemma, for which we adopt the convention $0\cdot
\infty\doteq0$. This is in fact the key link, showing that the rate function
that is naturally obtained by the weak convergence analysis used to
prove the
upper bound in fact equals $I$ for suitable measures, and also
indicating how
to construct controls to prove the lower bound for this same collection of
measures. Note that the constraints appearing in the lemma hold for the
subsequence appearing in (\ref{lower2}) due to Lemmas \ref{re1}, \ref
{re2} and
\ref{eeta}.

\begin{lemma}
\label{DV rate}Let $I ( \eta ) $ be defined by (\ref{I1}). Suppose
that $\eta\ll\pi$, that $\mu$ and $\xi$ satisfy the constraints
%
\begin{equation}
q ( x ) \eta ( dx ) =[\xi]_{1} ( dx ) \quad\mbox{and}\quad\int
_{%
\mathbb{R}
_{+}}u\xi ( dx,du ) =A [ \mu ] _{1} ( dx ),
\label{constr4}%
\end{equation}
and that when $A>0$ the constraint $ [ \mu ] _{1}= [
\mu ] _{2}$ is also true. Then
%
\begin{equation}
I ( \eta ) \leq AR \bigl( \mu\llVert [ \mu ] _{1}\otimes\alpha
\bigr) +\int_{S\times%
\mathbb{R}
_{+}}\ell ( u ) \xi ( dx,du ).
\label{eta_R_l}%
\end{equation}
Moreover,
\[
I ( \eta ) =\inf \biggl[ AR \bigl( \mu\llVert [ \mu ] _{1}\otimes
\alpha \bigr) +\int_{S\times%
\mathbb{R}
_{+}}\ell ( u ) \xi ( dx,du )
\biggr],
\]
where the infimum is over all possible choices of $A\geq0$, $\mu$ and
$\xi$
satisfying these constraints.
\end{lemma}

The proof of this lemma is detailed. The reason we present it here
instead of
in an \hyperref[app]{Appendix} is the previously mentioned fact that the construction
of $A$,
$\mu$ and $\xi$ that minimize the right-hand side of (\ref{eta_R_l}) indicates
how to hit target measures $\eta$ that are absolutely continuous with respect
to the invariant measure in the proof of the Laplace lower bound.

\begin{pf}
We first prove the inequality (\ref{eta_R_l}). If the right-hand side of
(\ref{eta_R_l}) is $\infty$, there is nothing to prove. Hence, we
assume it is
finite. First, assume $A>0$, in which case $R ( \mu\llVert  [
\mu ] _{1}\otimes\alpha  ) <\infty$. Define
%
\begin{equation}
Q\doteq\int_{S}q ( x ) \pi ( dx ), \label{q}%
\end{equation}
so that by (\ref{pi_pi_tilde})
%
\begin{equation}
\widetilde{\pi} ( dx ) =q ( x ) \pi ( dx ) /Q. \label{q_pi}%
\end{equation}
Since $\widetilde{\pi}$ is invariant under $\alpha$,
by \cite{dupell4}, Lemma~8.6.2, $ [ \mu ] _{1}\ll\widetilde{\pi}$. By (\ref{qk}),
$q$ is bounded from below, and hence $ [ \mu ] _{1}\ll\pi$. Recall
that the definition of $I$ in (\ref{I1}) uses $\theta=d\eta/d\pi$. Define
$\Theta\doteq \{ x\in S\dvtx\theta ( x ) =0 \} $. By
(\ref{constr4}),
%
\begin{eqnarray}
\label{mu_pi}%
[ \mu ] _{1} ( dx ) & =&\frac{\int_{%
\mathbb{R}
_{+}}u\xi_{2|1} ( du|x ) }{A} [
\xi ] _{1} ( dx )
\nonumber
\\
& =&\frac{\int_{%
\mathbb{R}
_{+}}u\xi_{2|1} ( du|x ) }{A}q ( x ) \eta ( dx )
\\
& =&\frac{\int_{%
\mathbb{R}
_{+}}u\xi_{2|1} ( du|x ) }{A}q ( x ) \theta ( x ) \pi ( dx ),\nonumber
\end{eqnarray}
where for a measure $\nu$ on $S\times%
\mathbb{R}
_{+}$, $\nu_{2|1}$ denotes the regular conditional distribution on the second
argument given the first. Thus, $ [ \mu ] _{1} ( \Theta
)
=0$. Now suppose that
\[
\int_{S\times S}\theta^{1/2} ( x ) \theta^{1/2} ( y
) q ( x ) \alpha ( x,dy ) \pi ( dx ) =0.
\]
Then for $\pi$-a.e. $x\in S\setminus\Theta$, $\alpha (
x,\Theta ) =1$, and hence $(\mu_{1}\otimes\alpha)[ ( S\setminus
\Theta ) \times\Theta]=1$. On the other hand, $\mu (  (
S\setminus\Theta ) \times\Theta ) =0$ due to $ [ \mu ]
_{1}= [ \mu ] _{2}$. This violates the fact that $R (
\mu\llVert  [ \mu ] _{1}\otimes\alpha  ) <\infty$.
We conclude that
\[
\int_{S\times S}\theta^{1/2} ( x ) \theta^{1/2} ( y
) q ( x ) \alpha ( x,dy ) \pi ( dx ) >0.
\]
Lemma~\ref{var_for} implies that
%
\begin{eqnarray}
\label{mu_pi_tilde}%
&&-\log \int_{S\times S}
\theta^{1/2} ( x ) \theta^{1/2} ( y ) \alpha ( x,dy ) \widetilde{
\pi} ( dx )
\nonumber
\\
&&\qquad =-\log\int_{S\times S}e^{({1}/{2}) [ \log\theta ( x )
+\log\theta ( y )  ] }\alpha ( x,dy ) \widetilde{
\pi} ( dx )
\\
&&\qquad \leq R ( \mu\llVert \widetilde{\pi}\otimes\alpha ) -
\frac{1}{2}\int_{S\times S} \bigl[ \log\theta ( x ) +\log
\theta ( y ) \bigr] \mu ( dx,dy ).\nonumber
\end{eqnarray}
Strictly speaking, the inequality above does not fall into the
framework of
Lemma~\ref{var_for} because $\log\theta$ is not bounded. However, if
one goes
through the proof of this lemma (\cite{dupell4}, Proposition~1.4.2), then the
above inequality is true as long as the right-hand side is not of the form
$\infty-\infty$. Toward this end, it suffices to prove
%
\begin{equation}\qquad
\frac{1}{2}\int_{S\times S} \bigl[ \log\theta ( x ) +\log
\theta ( y ) \bigr] \mu ( dx,dy ) =\int_{S}\log \theta ( x )
[ \mu ] _{1} ( dx ) <\infty. \label{theta_mu}%
\end{equation}
In the \hyperref[app]{Appendix}, we will prove [this being the only place where Condition
\ref{cond2'} is used] that
%
\begin{equation}
R \bigl( [ \mu ] _{1}\llVert \widetilde{\pi} \bigr) <\infty.
\label{inequality}%
\end{equation}
For now, we assume this is true. Using (\ref{q}), (\ref{q_pi}) and
(\ref{mu_pi}) to evaluate the relative entropy,
%
\begin{eqnarray}\label{u_pi_tilde}%
\infty>R \bigl( [ \mu ] _{1}\llVert \widetilde{\pi} \bigr) &=&
\int_{S}\log \biggl( \int_{%
\mathbb{R}
_{+}}u
\xi_{2|1} ( du|x ) \biggr) [ \mu ] _{1} ( dx )
\nonumber
\\[-8pt]
\\[-8pt]
\nonumber
&&{}+\int
_{S}\log\theta ( x ) [ \mu ] _{1} ( dx ) +\log
\frac{Q}{A}.
\end{eqnarray}
We know from (\ref{qk}) that $Q\geq K_{1}$. Also, by (\ref{constr4})
and the
nonnegativity of~$\ell$
%
\begin{eqnarray}\label{u_eta}
& &\int_{S}\log \biggl( \int_{%
\mathbb{R}
_{+}}u
\xi_{2|1} ( du|x ) \biggr) [ \mu ] _{1} ( dx )
\nonumber
\\
&&\qquad =\frac{1}{A}\int_{S} \biggl( \int
_{%
\mathbb{R}
_{+}}u\xi_{2|1} ( du|x ) \biggr) \log \biggl( \int
_{%
\mathbb{R}
_{+}}u\xi_{2|1} ( du|x ) \biggr) [ \xi ]
_{1} ( dx )
\nonumber
\\
&&\qquad =\frac{1}{A}\int_{S}\ell \biggl( \int
_{%
\mathbb{R}
_{+}}u\xi_{2|1} ( du|x ) \biggr) [ \xi ]
_{1} ( dx )
\nonumber
\\[-8pt]
\\[-8pt]
\nonumber
&&\qquad\quad{}+\frac{1}{A}\int_{S\times%
\mathbb{R}
_{+}}u\xi (
dx,du ) -\frac{1}{A}\int_{S} [ \xi ] _{1} (
dx )
\\
&&\qquad =\frac{1}{A}\int_{S}\ell \biggl( \int
_{%
\mathbb{R}
_{+}}u\xi_{2|1} ( du|x ) \biggr) [ \xi ]
_{1} ( dx ) +\int_{S} [ \mu ] _{1} ( dx
) -\frac{1}%
{A}\int_{S}q ( x ) \eta ( dx )\nonumber
\\
&&\qquad \geq1-\frac{1}{A}K_{2}.
\nonumber
\end{eqnarray}
The second constraint in (\ref{constr4}) is used for the first
equality; the
definition of $\ell$ gives the second equality; both parts of (\ref{constr4})
assure the third equality; finally the nonnegativity of $\ell$ is used. Thus,
rearranging (\ref{u_pi_tilde}) gives (\ref{theta_mu}).

The chain rule of relative entropy gives
%
\begin{eqnarray}
\label{mu_theta}%
&& R ( \mu\llVert \widetilde{\pi}\otimes\alpha
 ) -\frac{1}{2}\int_{S\times S} \bigl[ \log\theta
( x ) +\log \theta ( y ) \bigr] \mu ( dx,dy )
\nonumber
\\
&&\qquad =R \bigl( [ \mu ] _{1}\llVert \widetilde{\pi} \bigr) +\int
_{S}R ( \mu_{2|1}\llVert \alpha )
[ \mu ] _{1} ( dx ) -\int_{S}\log\theta ( x ) [
\mu ] _{1} ( dx )
\\
&&\qquad =R \bigl( [ \mu ] _{1}\llVert \widetilde{\pi} \bigr) +R
\bigl( \mu\llVert [ \mu ] _{1}\otimes\alpha \bigr) -\int
_{S}\log\theta ( x ) [ \mu ] _{1} ( dx ).\nonumber
\end{eqnarray}
By (\ref{u_pi_tilde}) and (\ref{u_eta}) and the convexity of $\ell$
%
\begin{eqnarray}
\label{xi_q} && R \bigl( [ \mu ] _{1}\llVert \widetilde{\pi}
\bigr) -\int_{S}\log\theta ( x ) [ \mu ] _{1} ( dx
)
\nonumber
\\
&&\qquad =\int_{S}\log \biggl( \int_{%
\mathbb{R}
_{+}}u
\xi_{2|1} ( du|x ) \biggr) [ \mu ] _{1} ( dx ) +\log
\frac{Q}{A}
\nonumber
\\
&&\qquad =\frac{1}{A}\int_{S}\ell \biggl( \int
_{%
\mathbb{R}
_{+}}u\xi_{2|1} ( du|x ) \biggr) [ \xi ]
_{1} ( dx )\\
&&\qquad\quad{} +\int_{S} [ \mu ] _{1} ( dx
) -\frac{1}%
{A}\int_{S}q ( x ) \eta ( dx ) +\log
\frac{Q}%
{A}
\nonumber\\
&&\qquad \leq\frac{1}{A}\int_{S\times%
\mathbb{R}
_{+}}\ell ( u ) \xi ( dx,du )
+1-\frac{1}{A}\int_{S}q ( x ) \eta ( dx ) +\log
\frac{Q}{A}.\nonumber 
\end{eqnarray}
In summary (\ref{mu_pi_tilde}), (\ref{mu_theta}) and (\ref{xi_q}) imply
\begin{eqnarray*}
&&-\log \int_{S\times S}\theta^{1/2} ( x )
\theta^{1/2} ( y ) q ( x ) \alpha ( x,dy ) \pi ( dx )
\\
&&\qquad =-\log\int_{S\times S}\theta^{1/2} ( x )
\theta^{1/2} ( y ) \alpha ( x,dy ) \widetilde{\pi} ( dx ) -\log Q
\\
&&\qquad \leq R \bigl( \mu\llVert [ \mu ] _{1}\otimes\alpha \bigr) +
\frac{1}{A}\int_{S\times%
\mathbb{R}
_{+}}\ell ( u ) \xi ( dx,du )\\
&&\qquad\quad{} +1-
\frac{1}{A}\int_{S}q ( x ) \eta ( dx ) +\log
\frac{1}{A}.
\end{eqnarray*}
Thus,
\begin{eqnarray*}
&& -\int_{S\times S}\theta^{1/2} ( x ) \theta^{1/2}
( y ) q ( x ) \alpha ( x,dy ) \pi ( dx )
\nonumber
\\[-8pt]
\\[-8pt]
\nonumber
&&\qquad \leq-\exp \biggl\{ - \biggl( R \bigl( \mu\llVert [ \mu ] _{1}\otimes
\alpha \bigr) +\frac{1}{A}\int_{S\times%
\mathbb{R}
_{+}}\ell ( u )
\xi ( dx,du ) \\
&&\hspace*{92pt}\qquad\quad{}+1-\frac{1}{A}\int_{S}q ( x ) \eta ( dx
) +\log\frac{1}{A} \biggr) \biggr\}.
\end{eqnarray*}
Equation (\ref{eta_R_l}) then follows from the fact that $-e^{-r}\leq ar+a\log
a-a$ for
any $r\in%
\mathbb{R}
$ and $a\in%
\mathbb{R}
_{+}$ by taking $a=A$ and
\[
r=R \bigl( \mu\llVert [ \mu ] _{1}\otimes\alpha \bigr) +
\frac{1}{A}\int_{S\times%
\mathbb{R}
_{+}}\ell ( u ) \xi ( dx,du ) +1-
\frac{1}{A}\int_{S}q ( x ) \eta ( dx ) +\log
\frac{1}{A}.
\]
For the case when $A=0$, (\ref{constr4}) implies that $\int_{%
\mathbb{R}
_{+}}u\xi ( dx,du ) =0$, which means that $\int_{%
\mathbb{R}
_{+}}u\xi_{2|1} ( du|x ) =0$ $ [ \xi ] _{1}$-a.e. Hence,
by the convexity of $\ell$ and $q ( x ) \eta ( dx )
=[\xi]_{1} ( dx ) $,
\begin{eqnarray*}
\int_{S\times%
\mathbb{R}
_{+}}\ell ( u ) \xi ( dx,du ) & \geq&\int
_{S}%
\ell \biggl( \int_{%
\mathbb{R}
_{+}}u
\xi_{2|1} ( du|x ) \biggr) [ \xi ] _{1} ( dx )
\\
& =&\int_{S} [ \xi ] _{1} ( dx )
\\
& =&\int_{S}q ( x ) \eta ( dx )
\\
& \geq&\int_{S}q ( x ) \eta ( dx ) -\int
_{S\times
S}\theta^{1/2} ( x ) \theta^{1/2} ( y ) q
( x ) \alpha ( x,dy ) \pi ( dx )
\\
& =&I ( \eta ).
\end{eqnarray*}
Thus, (\ref{eta_R_l}) also holds in this case, and completes the proof
of the
first part of Lemma~\ref{DV rate}.

We now turn to the second part of Lemma~\ref{DV rate}. The definitions and
constructions used here will also be used to construct what are essentially
optimal controls to prove the reverse inequality in the next section, and
indeed the particular forms of the definitions are suggested by that
use. In
particular, $A\kappa(x)$ will correspond to a dilation of the mean for the
exponential random variables. In light of the second part of Lemma~\ref{var_for}, we define $\mu$ by
%
\begin{eqnarray}\label{mu}%
&&\frac{d\mu}{d ( \widetilde{\pi}\otimes\alpha ) } ( x,y )
\nonumber
\\[-8pt]
\\[-8pt]
\nonumber
&&\qquad \doteq\theta^{1/2} ( x ) \theta^{1/2}
( y ) \Big/ \int_{S\times S}\theta^{1/2} ( x )
\theta^{1/2} ( y ) ( \widetilde{\pi}\otimes\alpha ) ( dx,dy ) .
\end{eqnarray}
Note that by the Cauchy--Schwarz inequality, the detailed balance condition
(\ref{reversible2}) and the relation between $\pi$ and $\widetilde{\pi
}$ [see
(\ref{pi_pi_tilde})] imply
\[
\int_{S\times S}\theta^{1/2} ( x ) \theta^{1/2} ( y
) ( \widetilde{\pi}\otimes\alpha ) ( dx,dy ) \leq \int_{S\times S}
\theta ( x ) \alpha ( x,dy ) \widetilde{\pi} ( dx ) \leq\frac{K_{2}}{Q}.
\]
Hence, $\mu$ is well defined and $ [ \mu ] _{1}= [ \mu
]
_{2}$. Then Lemma~\ref{var_for} implies that
%
\begin{eqnarray}\label{theta_eq}%
&&-\log\int_{S\times S}\theta^{1/2} ( x )
\theta^{1/2} ( y ) \alpha ( x,dy ) \widetilde{\pi} ( dx )
\nonumber
\\[-8pt]
\\[-8pt]
\nonumber
&&\qquad =R ( \mu
\llVert \widetilde{\pi}\otimes\alpha ) -\int_{S}
\log\theta ( x ) [ \mu ] _{1} ( dx ).
\end{eqnarray}

If $R ( \mu\llVert \widetilde{\pi}\otimes\alpha  )
=\infty$ or $-\int_{S}\log\theta ( x )  [ \mu ]
_{1} ( dx ) =\infty$, the last display implies
\[
\int_{S\times S}\theta^{1/2} ( x ) \theta^{1/2} ( y
) q ( x ) \alpha ( x,dy ) \pi ( dx ) =0.
\]
By letting $A\doteq0$ and $\xi ( dx,du ) \doteq q ( x )
\eta ( dx ) \delta_{0} ( du ) $, then $\xi$ and $\mu$
satisfy (\ref{constr4}) and
\[
AR \bigl( \mu\llVert [ \mu ] _{1}\otimes\alpha \bigr) +\int
_{S\times%
\mathbb{R}
_{+}}\ell ( u ) \xi ( dx,du ) =\int_{S}q
( x ) \eta ( dx ) =I ( \eta ).
\]

Next, assume $R ( \mu\llVert \widetilde{\pi}\otimes\alpha
) <\infty$ and $-\int_{S}\log\theta ( x )  [
\mu ] _{1} ( dx ) <\infty$. Define $A$ by
%
\begin{equation}
A\doteq\exp \biggl\{ - \biggl[ R ( \mu\llVert \widetilde{\pi}%
\otimes\alpha ) -\int_{S}\log\theta ( x ) [ \mu
] _{1} ( dx ) -\log Q \biggr] \biggr\}. \label{A}%
\end{equation}
Define the measure
%
\begin{equation}
\rho ( dx ) \doteq q ( x ) \theta ( x ) \pi ( dx ) \label{rho_con}%
\end{equation}
and
%
\begin{equation}
\kappa\doteq d [ \mu ] _{1}/d\rho. \label{k_con}%
\end{equation}
Then for any $x\in S\setminus\Theta$ (recall $\Theta= \{ x\in
S\dvtx\theta(x)=0 \} $)
\[
\kappa ( x ) =\frac{d [ \mu ] _{1}}{d\rho} ( x ) =\frac{1}{Q\theta ( x ) }\frac{d [ \mu ] _{1}
}{d\widetilde{\pi}} ( x
).
\]
In addition,
%
\begin{eqnarray}
\label{f(x)}%
\qquad\int_{S}\kappa ( x ) \log\kappa (
x ) \rho ( dx ) & =&\int_{S}\log\kappa ( x ) [ \mu ]
_{1} ( dx )
\nonumber
\\[-8pt]
\\[-8pt]
\nonumber
& =&R \bigl( [ \mu ] _{1}\llVert \widetilde{\pi} \bigr) -\int
_{S}\log\theta ( x ) [ \mu ] _{1} ( dx ) -\log Q.
\end{eqnarray}
Define
%
\begin{equation}
b ( x ) \doteq\cases{ %
0, &\quad
$\mbox{for }x\in\Theta,$
\vspace*{2pt}\cr
A\kappa ( x ), & \quad$\mbox{for }x\notin\Theta,$}
\label{b}%
\end{equation}
and
%
\begin{equation}
\xi ( dx,du ) \doteq q ( x ) \eta ( dx ) \delta_{b ( x ) } ( du ).
\label{xi_con}%
\end{equation}
Then $\xi$ satisfies the first part of (\ref{constr4}). To see that the second
part of (\ref{constr4}) is satisfied, note that
\[
[ \mu ] _{1} ( \Theta ) =0=\int_{\Theta\times%
\mathbb{R}
_{+}}u\xi ( dx,du
)
\]
and
\begin{eqnarray*}
\int_{%
\mathbb{R}
_{+}}u\xi ( dx,du ) & =&b ( x ) q ( x ) \eta ( dx )
\\
& =&A\kappa ( x ) q ( x ) \theta ( x ) \pi ( dx )
\\
& =&A [ \mu ] _{1} ( dx ).
\end{eqnarray*}
By using the definitions we arrive at the following, each line of which is
explained below:
\begin{eqnarray*}
&& AR \bigl( \mu\llVert [ \mu ] _{1}\otimes\alpha \bigr) +\int
_{S\times\mathbb{R}_{+}}\ell ( u ) \xi ( dx,du )
\\
&&\qquad =AR \bigl( \mu\llVert [ \mu ] _{1}\otimes\alpha \bigr) +\int
_{S}\ell \bigl( b ( x ) \bigr) q ( x ) \eta ( dx )
\\
&&\qquad =AR \bigl( \mu\llVert [ \mu ] _{1}\otimes\alpha \bigr) +\int
_{\Theta}q ( x ) \eta ( dx ) +\int_{S\setminus\Theta}\ell
\bigl( b ( x ) \bigr) \rho ( dx )
\\
&&\qquad =AR \bigl( \mu\llVert [ \mu ] _{1}\otimes\alpha \bigr) +\int
_{S}q ( x ) \eta ( dx ) +A\log A-A\\
&&\qquad\quad{}+A\int_{S}
\kappa ( x ) \log\kappa ( x ) \rho ( dx )
\\
&&\qquad =\int_{S}q ( x ) \eta ( dx ) +A\log A-A\\
&&\qquad\quad{}+A \biggl( R
( \mu\llVert \widetilde{\pi}\otimes\alpha ) -\int
_{S}\log\theta ( x ) [ \mu ] _{1} ( dx ) -\log Q
\biggr)
\\
&&\qquad =\int_{S}q ( x ) \eta ( dx ) -A.
\end{eqnarray*}
The first equality uses (\ref{xi_con}) and the second uses (\ref{b}). The
third uses (\ref{b}) again, expands $\ell$, and uses $\kappa\doteq
d [
\mu ] _{1}/d\rho$ and $\eta ( \Theta ) =\rho (
\Theta ) =0$. Equality four then uses (\ref{f(x)}) and the fifth follows
from (\ref{A}). Note that (\ref{theta_eq}) and (\ref{pi_pi_tilde}) imply
%
\begin{equation}
A=\int_{S\times S}\theta^{1/2} ( x ) \theta^{1/2} (
y ) q ( x ) \alpha ( x,dy ) \pi ( dx ). \label{a}%
\end{equation}
Hence, we obtain
\[
AR \bigl( \mu\llVert [ \mu ] _{1}\otimes\alpha \bigr) +\int
_{S}\ell ( u ) \xi ( dx,du ) =I ( \eta ) .
\]
\upqed\end{pf}

The representation formula (\ref{rep_gen}), the lower bound (\ref
{lower2}) and
Lemma~\ref{DV rate} together give
%
\begin{eqnarray}\label{final_lb}
&& \liminf_{T\rightarrow\infty}-\frac{1}{T}\log E \bigl[ \exp \bigl\{
-TF(\eta_{T})-T\cdot\infty\cdot \bigl( 1_{ \{ r_{T}/T \} ^{c}%
} (
R_{T}/T ) \bigr) \bigr\} \bigr]
\nonumber
\\[-8pt]
\\[-8pt]
\nonumber
&&\qquad \geq\inf_{\eta\in\mathcal{P} ( S ) } \bigl[ F ( \eta ) +I ( \eta ) \bigr].
\end{eqnarray}

\subsection{Combining the cases}

In the last section, we showed that (\ref{final_lb}) is valid for any sequence
$ \{ r_{T} \} $ such that $r_{T}/T\rightarrow A$ $\in [
0,C ] $. An argument by contradiction shows that the bound is
uniform in
$A$. Thus,
\begin{eqnarray*}
&& \liminf_{T\rightarrow\infty}-\frac{1}{T}\log \Biggl\{ \sum
_{r_{T}%
=1}^{ \lfloor TC \rfloor}E \bigl[ \exp \bigl\{ -TF(
\eta_{T}%
)-T\cdot\infty\cdot \bigl( 1_{ \{ r_{T}/T \} ^{c}} (
R_{T}/T ) \bigr) \bigr\} \bigr] \Biggr\}
\\
&&\qquad \geq\liminf_{T\rightarrow\infty}-\frac{1}{T}\log \Biggl\{ TC\cdot
\bigvee_{r_{T}=1}^{ \lfloor TC \rfloor}E \bigl[ \exp \bigl\{ -TF(
\eta_{T})\\
&&\qquad\quad\hspace*{149pt}{}-T\cdot\infty\cdot \bigl( 1_{ \{ r_{T}/T \} ^{c}%
} ( R_{T}/T
) \bigr) \bigr\} \bigr] \Biggr\}
\\
&&\qquad \geq\inf_{\eta\in\mathcal{P} ( S ) } \bigl[ F ( \eta ) +I ( \eta ) \bigr].
\end{eqnarray*}
We now partition $E [ \exp \{ -TF(\eta_{T}) \}  ] $
according to the two cases to obtain the overall lower bound
\begin{eqnarray*}
&&\liminf_{T\rightarrow\infty}  -\frac{1}{T}\log E \bigl[ \exp \bigl\{
-TF(\eta_{T}) \bigr\} \bigr]
\\
&&\qquad \geq\min \Bigl\{ \inf_{\eta\in\mathcal{P} ( S ) } \bigl[ F ( \eta ) +I ( \eta )
\bigr], [ -C+C\log C+K_{2}-C\log K_{2} ] \Bigr\}.
\end{eqnarray*}
Letting $C\rightarrow\infty$ we have the desired Laplace upper bound
%
\begin{equation}
\liminf_{T\rightarrow\infty}-\frac{1}{T}\log E \bigl[ \exp \bigl\{ -TF(
\eta_{T}) \bigr\} \bigr] \geq\inf_{\eta\in\mathcal{P} ( S )
} \bigl[ F (
\eta ) +I ( \eta ) \bigr]. \label{lp_upper}%
\end{equation}

\section{Proof of Laplace lower bound}
\label{lplb}

We turn to the proof of the reverse inequality
%
\begin{equation}
\limsup_{T\rightarrow\infty}-\frac{1}{T}\log E \bigl[ \exp \bigl\{ -TF(
\eta_{T}) \bigr\} \bigr] \leq\inf_{\eta\in\mathcal{P} ( S )
} \bigl[ F (
\eta ) +I ( \eta ) \bigr]. \label{lb}%
\end{equation}
Let $F$ be a nonnegative bounded and continuous function. Fix an arbitrary
$\varepsilon>0$ and choose $\eta$ such that
%
\begin{equation}
F ( \eta ) +I ( \eta ) \leq\inf_{\nu\in\mathcal{P}%
( S ) } \bigl[ F ( \nu ) +I (
\nu ) \bigr] +\varepsilon. \label{eta_1}%
\end{equation}
As pointed out in Remark~\ref{isolate}, $H$ defined in (\ref{H}) is
dense in
$\mathcal{P} ( S ) $. Since $I$ was extended from $H$ to
$\mathcal{P} ( S ) $ via lower semicontinuous regularization, we
can assume without loss of generality that $\eta\ll\pi$. Define $\theta
\doteq
d\eta/d\pi$. We now argue we can further assume there exists $\delta>0$ such
that
%
\begin{equation}
\delta\leq\theta ( x ) \leq\frac{1}{\delta} \label{eta_2}%
\end{equation}
for all $x\in S$. If $\eta^{\delta}\doteq ( 1-\delta ) \eta
+\delta\pi$ then $d\eta^{\delta}/d\pi\geq\delta$, and the continuity of $F$
and the convexity of $I$ imply that the difference between $F (
\eta^{\delta} ) +I ( \eta^{\delta} ) $ and $F (
\eta ) +I ( \eta ) $ can be made arbitrarily small.

Thus, we can assume $\theta$ is uniformly bounded from below away from zero.
Let $n\in\mathbb{N}$, and define
\[
\eta^{n} ( dx ) \doteq\theta ( x ) 1_{ \{
\theta ( x ) \leq n \} }\pi ( dx ) +
\frac
{\eta (  \{ x\dvtx\theta ( x ) >n \}  ) }%
{\pi (  \{ x\dvtx\theta ( x ) >n \}  )
}1_{ \{ \theta ( x ) >n \} }\pi ( dx ).
\]
Then $d\eta^{n}/d\pi\leq{}[\eta (  \{ x\dvtx\theta (
x )
>n \}  ) /\pi (  \{ x\dvtx\theta ( x )
>n \}  ) ]\vee n$, and since $\eta\ll\pi$ implies $\pi (
\{ x\dvtx\theta ( x ) >n \}  ) \rightarrow0$,
$\eta^{n}$ converges weakly to $\eta$. It then follows from the
continuity of
$F$ and the definition of $I$ and convexity of $\theta\rightarrow-\theta
^{1/2}$ that we can choose $\eta$ satisfying (\ref{eta_1}) with
$2\varepsilon$
replacing $\varepsilon$ and also (\ref{eta_2}). Hence, we assume $\eta$
satisfies (\ref{eta_1}) and (\ref{eta_2}). Furthermore, by Lusin's theorem
(\cite{fol}, Theorem~7.10), we can also assume that $\theta$ is continuous.

The proof of the lower bound will use the following representation. The
infimum in the representation is taken over all control measures $
\{
\bar{\alpha}_{i},\bar{\sigma}_{i} \} $, and the properties of such
measures and how $\bar{\eta}_{T}$ and $\bar{R}_{T}$ are constructed
from them
were discussed immediately above the similar representation (\ref{rep_gen}).
The proof of the lemma is given in the \hyperref[app]{Appendix}.

\begin{lemma}
\label{rep_ran}Let $F\dvtx\mathcal{P} ( S ) \rightarrow%
\mathbb{R}
$ be bounded and continuous. Then
\[
-\frac{1}{T}\log E \bigl[ \exp \bigl\{ -TF ( \eta_{T} ) \bigr\}
\bigr] =\inf E \Biggl[ F ( \bar{\eta}_{T} ) +\frac{1}{T}\sum
_{i=1}^{\bar{R}_{T}} \bigl( R ( \bar{
\alpha}_{i-1}\llVert \alpha ) +R ( \bar{
\sigma}_{i}\llVert \sigma ) \bigr) \Biggr],
\]
where the infimum is taken over all control measures $ \{ \bar
{\alpha
}_{i},\bar{\sigma}_{i} \} $.
\end{lemma}

Suppose that given any measure $\eta\in\mathcal{P} ( S ) $
satisfying (\ref{eta_1}) and (\ref{eta_2}), one can construct $\bar
{\alpha
}_{i}$ and $\bar{\sigma}_{i}$ such that given any subsequence of $T$,
there is
a further subsequence $T_{n}$ such that
\[
\lim_{T_{n}\rightarrow\infty}E \Biggl[ F ( \bar{\eta}_{T_{n}} ) +
\frac{1}{T_{n}}\sum_{i=1}^{\bar{R}_{T_{n}}} \bigl( R
( \bar{\alpha }_{i-1}\llVert \alpha ) +R ( \bar{
\sigma}_{i}\llVert \sigma ) \bigr) \Biggr] =F ( \eta ) +I
( \eta ).
\]
Then Lemma~\ref{rep_ran} implies the Laplace lower bound (\ref{lb}). The
construction of suitable $\bar{\alpha}_{i}$ and $\bar{\sigma}_{i}$
turns on
many of the same constructions as those used in the proof of the second part
of Lemma~\ref{DV rate}. We first define $\mu\in\mathcal{P} ( S\times
S ) $ as in (\ref{mu}). Then automatically $ [ \mu ]
_{1}= [ \mu ] _{2}$, and hence if we define $p$ as the regular
conditional probability such that $\mu= [ \mu ] _{1}\otimes p$,
then $ [ \mu ] _{1}$ is invariant under $p$ (\cite{dupell4},
Lemma~8.5.1(a)). Define $\bar{\alpha}_{i}\doteq p$ for each $i$, and let
$ \{ \bar{X}_{i} \} $ be the corresponding Markov chain. Next,
define $\rho ( dx ) \doteq q ( x ) \eta (
dx ) $ and
%
\begin{equation}
\kappa ( x ) \doteq\frac{d [ \mu ] _{1}}{d\rho} ( x ) =\frac{1}{Q\theta ( x ) }
\frac{d [ \mu ] _{1}
}{d\widetilde{\pi}} ( x ). \label{k}%
\end{equation}
By (\ref{eta_2}), there is $M<\infty$ such that $1/M\leq\kappa\leq M$,
and due
to the continuity of $\theta$, $\kappa$ is also continuous. Notice that
%
\begin{equation}
\eta ( dx ) = \bigl( q ( x ) \kappa ( x ) \bigr) ^{-1} [ \mu ]
_{1} ( dx ). \label{eta_mu1}%
\end{equation}
Assumption (\ref{eta_2}) guarantees that
\begin{eqnarray*}
-\log\int_{S\times S}\theta^{1/2} ( x )
\theta^{1/2} ( y ) ( \widetilde{\pi}\otimes\alpha ) ( dx,dy ) &<&\infty
\quad\mbox{and}\\
-\int_{S}\log\theta ( x ) [ \mu ] _{1} (
dx ) &<&\infty,
\end{eqnarray*}
and (\ref{theta_eq}) then implies that $R ( \mu\llVert \widetilde
{\pi
}\otimes\alpha  ) <\infty$. Define $A$ as in (\ref{a}). Let
$\bar{\sigma}_{i}$ be the exponential distribution with mean $ [
A\kappa ( \bar{X}_{i-1} )  ] ^{-1}$ for each $i$. Thus,
we can
construct a Markov jump process $\bar{X} ( t ) $ using $\bar
{\alpha}_{i}$ and $\bar{\sigma}_{i}$ instead of $\alpha$ and $\sigma$,
and the
infinitesimal $\bar{\mathcal{L}}$ generator will be bounded and continuous
and takes the form:
\[
\bar{\mathcal{L}}f ( x ) =A\kappa ( x ) q ( x ) \int_{S}
\bigl[ f ( y ) -f ( x ) \bigr] p ( x,dy ).
\]
Equation~(\ref{eta_mu1}) and the fact that $ [ \mu ] _{1}$ is invariant
under $p$ imply $\int_{S} ( \bar{\mathcal{L}}f ( x )  )\*
\eta ( dx ) =0$, and $\eta$ is an invariant distribution of the
continuous time process $\bar{X}$. We claim that $\eta$ is the unique
invariant distribution of $\bar{X}$. Indeed, by \cite{ethkur}, Proposition~4.9.2, any invariant distribution $\nu$ for $\bar{X}$ satisfies
$\int_{S} ( \bar{\mathcal{L}}f ( x )  ) \nu (
dx ) =0$. If we define
\[
\widetilde{\nu} ( dx ) \doteq\frac{A\kappa ( x )
q ( x ) \nu ( dx ) }{\int_{S}A\kappa ( x )
q ( x ) \nu ( dx ) },
\]
then $\widetilde{\nu}$ is invariant under $p$. However, by Condition
\ref{cond2} and \cite{dupell4}, Lemma~8.6.3(c), the invariant measure
under $p$
is unique, and hence the invariant measure of $\bar{X}$ is also unique.
By the
definition of $\bar{\eta}_{T}$ in (\ref{eta_T}),
%
\begin{eqnarray}
\label{eta_bar2}%
\qquad\bar{\eta}_{T} ( \cdot ) & =&\frac{1}{T}
\int_{0}^{T}\delta _{\bar{X} ( t ) } ( \cdot ) \,dt
\nonumber
\\[-8pt]
\\[-8pt]
\nonumber
& =&\frac{1}{T} \Biggl[ \sum_{i=1}^{\bar{R}_{T}-1}
\delta_{\bar
{X}_{i-1}} ( dx ) \frac{\bar{\tau}_{i}}{q ( \bar{X}_{i-1} ) }+\delta _{\bar{X}_{\bar{R}_{T}-1}} ( dx )
\Biggl( T-\sum_{i=1}^{\bar{R}%
_{T}-1}\frac{\bar{\tau}_{i}}{q ( \bar{X}_{i-1} ) }
\Biggr) \Biggr] .
\end{eqnarray}
Since $S$ is compact, we can extract a subsequence of $T$ such that
$\bar{\eta
}_{T}$ converges weakly, and by \cite{ethkur}, Theorem~4.9.3, this weak limit
is $\eta$. We claim the following along the same subsequence.

\begin{lemma}
\label{lem:costs}$E [ \bar{R}_{T}/T ] \rightarrow A$, $E [
\sum_{i=1}^{\bar{R}_{T}}R ( \bar{\sigma}_{i}\llVert \sigma
) /T ] \rightarrow\int_{S}\ell ( A\kappa ( x )
)  q ( x ) \eta ( dx ) $ and $E [ \sum_{i=1}^{\bar{R}_{T}}\delta_{\bar{X}_{i-1}} ( dx ) /T ]
\rightarrow A [ \mu ] _{1} ( dx ) $.
\end{lemma}

\begin{pf}
As in the proof of the upper bound, a minor nuisance is dealing with the
residual time $T-\sum_{i=1}^{\bar{R}_{T}}\bar{\tau}_{i}$. However, this is
more easily controlled here since it is bounded by an exponential with known
mean. Since $\bar{\eta}_{T}\rightarrow\eta$ weakly, we have for any bounded
and continuous function $f$ on the space of subprobability measures on $S$
that $\lim_{T\rightarrow\infty}E [ f ( \bar{\eta}_{T} )
] =f ( \eta ) $. To prove the first part of the lemma,
define $f$ by
\[
f ( \nu ) \doteq\int_{S}\kappa ( x ) q ( x ) \nu ( dx ).
\]
Since both $\kappa$ and $q$ are bounded and continuous, $f$ is also bounded
and continuous. Using (\ref{eta_mu1})
%
\begin{equation}
f ( \eta ) =\int_{S}\kappa ( x ) q ( x ) \eta ( dx ) =\int
_{S} [ \mu ] _{1} ( dx ) =1. \label{f1}%
\end{equation}
Thus, $\lim_{T\rightarrow\infty}E [ f ( \bar{\eta}_{T} )
] =1$. Now by (\ref{eta_bar2}) and the definition of $\bar{R}_{T}$
\begin{eqnarray*}
&& E \Biggl[ \Biggl\llvert f ( \bar{\eta}_{T} ) -f \Biggl(
\frac{1}%
{T}\sum_{i=1}^{\bar{R}_{T}}
\delta_{\bar{X}_{i-1}} ( dx ) \frac
{\bar{\tau}_{i}}{q ( \bar{X}_{i-1} ) } \Biggr) \Biggr\rrvert \Biggr]
\\
&&\qquad =\frac{1}{T}E \Biggl[ \kappa ( \bar{X}_{\bar{R}_{T}-1} ) q (
\bar{X}_{\bar{R}_{T}-1} ) \Biggl( \sum_{i=1}^{\bar{R}_{T}}
\frac
{\bar{\tau}_{i}}{q ( \bar{X}_{i-1} ) }-T \Biggr) \Biggr]
\\
&&\qquad \leq\frac{1}{T}E \biggl[ \kappa ( \bar{X}_{\bar{R}_{T}-1} ) q (
\bar{X}_{\bar{R}_{T}-1} ) \frac{\bar{\tau}_{\bar{R}_{T}}%
}{q ( \bar{X}_{\bar{R}_{T}-1} ) } \biggr]
\\
&&\qquad \leq\frac{M}{T}E [ \bar{\tau}_{\bar{R}_{T}} ]\leq\frac{AM^{2}}{T}\rightarrow0
\end{eqnarray*}
as $T\rightarrow\infty$. Hence,
\[
\lim_{T\rightarrow\infty}E \Biggl[ f \Biggl( \frac{1}{T}\sum
_{i=1}^{\bar
{R}_{T}%
}\delta_{\bar{X}_{i-1}} ( dx )
\frac{\bar{\tau}_{i}}{q (
\bar{X}_{i-1} ) } \Biggr) \Biggr] =1.
\]
Recall that $\mathcal{F}_{i}$ is the $\sigma$-algebra generated by
$ \{
( \bar{X}_{0},\ldots,\bar{X}_{i} ), ( \bar{\tau}_{1}%
,\ldots,\bar{\tau}_{i} )  \} $. Then
\begin{eqnarray*}
&& E \Biggl[ f \Biggl( \frac{1}{T}\sum_{i=1}^{\bar{R}_{T}}
\delta_{\bar
{X}_{i-1}%
} ( dx ) \frac{\bar{\tau}_{i}}{q ( \bar{X}_{i-1} )
} \Biggr) \Biggr]
\\
&&\qquad =\frac{1}{T}E \Biggl[ \sum_{i=1}^{\bar{R}_{T}}
\kappa ( \bar{X}%
_{i-1} ) q ( \bar{X}_{i-1} )
\frac{\bar{\tau}_{i}}{q (
\bar{X}_{i-1} ) } \Biggr]
\\
&&\qquad =\frac{1}{T}E \Biggl[ \sum_{i=1}^{\infty}
\kappa ( \bar {X}_{i-1} ) \bar{\tau}_{i}\mathbf{1} \Biggl(
\sum_{j=1}^{i-1}\frac{\bar{\tau}_{j}%
}{q ( \bar{X}_{j-1} ) }\leq T
\Biggr) \Biggr]
\\
&&\qquad =\frac{1}{T}\sum_{i=1}^{\infty}E
\Biggl[ E \Biggl[  \kappa ( \bar{X}_{i-1} ) \bar{
\tau}_{i}\mathbf{1} \Biggl( \sum_{j=1}%
^{i-1}\frac{\bar{\tau}_{j}}{q ( \bar{X}_{j-1} ) }\leq T \Biggr) \Big\rrvert \mathcal{F}_{i-1}
\Biggr] \Biggr]
\\
&&\qquad =\frac{1}{T}\sum_{i=1}^{\infty}E
\Biggl[ \kappa ( \bar {X}_{i-1} ) \mathbf{1} \Biggl( \sum
_{j=1}^{i-1}\frac{\bar{\tau}_{j}}{q ( \bar
{X}_{j-1} ) }\leq T \Biggr) E [ \bar{
\tau}_{i}|\mathcal{F}%
_{i-1} ] \Biggr]
\\
&&\qquad =\frac{1}{T}\sum_{i=1}^{\infty}E
\Biggl[ \kappa ( \bar {X}_{i-1} ) \mathbf{1} \Biggl( \sum
_{j=1}^{i-1}\frac{\bar{\tau}_{j}}{q ( \bar
{X}_{j-1} ) }\leq T \Biggr)
\frac{1}{A\kappa ( \bar{X}%
_{i-1} ) } \Biggr]
\\
&&\qquad =\frac{1}{AT}E \Biggl[ \sum_{i=1}^{\infty}
\mathbf{1} \Biggl( \sum_{j=1}^{i-1}
\frac{\bar{\tau}_{j}}{q ( \bar{X}_{j-1} ) }\leq T \Biggr) \Biggr]
\\
&&\qquad =\frac{1}{A}E \biggl[ \frac{\bar{R}_{T}}{T} \biggr].
\end{eqnarray*}
This completes the proof of the first statement in the lemma.

The proof of the second statement is similar. Define $f$ by
\[
f ( \nu ) \doteq\int_{S}\ell \bigl( A\kappa ( x ) \bigr) q
( x ) \nu ( dx ).
\]
Then as before,
\[
f ( \eta ) =\lim_{T\rightarrow\infty}E \bigl[ f ( \eta _{T} )
\bigr] =\lim_{T\rightarrow\infty}E \Biggl[ f \Biggl( \frac{1}%
{T}\sum
_{i=1}^{\bar{R}_{T}}\delta_{\bar{X}_{i-1}} ( dx )
\frac
{\bar{\tau}_{i}}{q ( \bar{X}_{i-1} ) } \Biggr) \Biggr].
\]
Using $g ( x ) =x\ell ( 1/x ) $ and Lemma~\ref{exp_re},
we have
\begin{eqnarray*}
&& E \Biggl[ f \Biggl( \frac{1}{T}\sum_{i=1}^{\bar{R}_{T}}
\delta_{\bar
{X}_{i-1}%
} ( dx ) \frac{\bar{\tau}_{i}}{q ( \bar{X}_{i-1} )
} \Biggr) \Biggr]
\\
& &\qquad=E \Biggl[ \frac{1}{T}\sum_{i=1}^{\bar{R}_{T}}
\ell \bigl( A\kappa ( \bar{X}_{i-1} ) \bigr) \bar{\tau}_{i}
\Biggr]
\\
&&\qquad =\frac{1}{T}\sum_{i=1}^{\infty}E
\Biggl[ \ell \bigl( A\kappa ( \bar {X}_{i-1} ) \bigr) \mathbf{1}
\Biggl( \sum_{j=1}^{i-1}\frac
{\bar{\tau}_{j}}{q ( \bar{X}_{j-1} ) }\leq
T \Biggr) E [ \bar{\tau}_{i}|\mathcal{F}_{i-1} ] \Biggr]
\\
&&\qquad =E \Biggl[ \frac{1}{T}\sum_{i=1}^{\bar{R}_{T}}
\frac{1}{A\kappa (
\bar{X}_{i-1} ) }\ell \bigl( A\kappa ( \bar{X}_{i-1} ) \bigr) \Biggr]
\\
&&\qquad =E \Biggl[ \frac{1}{T}\sum_{i=1}^{\bar{R}_{T}}g
\biggl( \frac{1}%
{A\kappa ( \bar{X}_{i-1} ) } \biggr) \Biggr]
\\
&&\qquad =E \Biggl[ \frac{1}{T}\sum_{i=1}^{\bar{R}_{T}}R
( \bar{\sigma}%
_{i}\llVert \sigma )
\Biggr],
\end{eqnarray*}
and the second part of the lemma follows.

The proof of the third part follows very similar lines as the first
two, and
is omitted.
\end{pf}

Now the Laplace lower bound is straightforward. The definition of $\mu$ in
(\ref{mu}), the continuity of $\theta$ and the bound (\ref{eta_2}) imply
$x\rightarrow R ( p ( x,\cdot ) \llVert \alpha (
x,\cdot )   ) $ is bounded and continuous. By Lemma~\ref{lem:costs} and the chain rule for relative entropy,
\begin{eqnarray*}
&& \lim_{T\rightarrow\infty}E \Biggl[ F ( \bar{\eta}_{T} ) +
\frac
{1}{T}\sum_{i=1}^{\bar{R}_{T}} \bigl( R
( \bar{\alpha}_{i-1}\llVert \alpha ) +R ( \bar{
\sigma}_{i}\llVert \sigma ) \bigr) \Biggr]
\\
&&\qquad =\lim_{T\rightarrow\infty}E \bigl[ F ( \bar{\eta}_{T} ) \bigr] +
\lim_{T\rightarrow\infty}\int_{S}R \bigl( p ( x,\cdot )
\llVert \alpha ( x,\cdot )  \bigr) E \Biggl[ \frac{1}{T}\sum
_{i=1}^{\bar{R}_{T}}\delta_{\bar{X}_{i-1}} ( dx )
\Biggr]
\\
&&\qquad\quad{} +\lim_{T\rightarrow\infty}E \Biggl[ \frac{1}{T}\sum
_{i=1}^{\bar{R}_{T}%
}R ( \bar{\sigma}_{i}\llVert
\sigma ) \Biggr]
\\
&&\qquad =F ( \eta ) +AR \bigl( \mu\llVert [ \mu ] _{1}\otimes\alpha
\bigr) +\int_{S}\ell \bigl( A\kappa ( x ) \bigr) q ( x ) \eta
( dx ).
\end{eqnarray*}
Returning to the proof of the second part of Lemma~\ref{DV rate}, we
find that
with this choice of $A$, $\mu$ and $\kappa$, the rate function $I (
\eta ) $ coincides with $AR ( \mu\llVert  [ \mu ]
_{1}\otimes\alpha  ) +\int_{S}\ell ( A\kappa (
x )  ) q ( x ) \eta ( dx ) $ (note that this
$\eta$ corresponds to a special of Lemma~\ref{DV rate} where $\Theta
\doteq \{ x\in S\dvtx\theta ( x ) =0 \} $ is empty). This
completes the proof of the Laplace lower bound.

\section{On the boundedness of rate function}

As pointed out in the \hyperref[intro]{Introduc-}\break\hyperref[intro]{tion}, continuous time jump Markov processes
differ from the type of processes considered by Donsker and Varadhan in
\cite{donvar1,donvar3}, in that the dynamics do not have a
``diffusive'' component, and hence Condition \ref{cond_diff}
does not hold. For jump Markov models, the process only moves when a jump
occurs, and there is no continuous change of position. For these
processes, the
rate function is bounded, whereas for the processes of \cite{donvar1,donvar3}
the rate function is infinity when the target measure is not absolutely
continuous with respect to the reference measure. We now consider the source
and implications of this distinction.

Consider a process satisfying all the conditions in Section~\ref{setup} that
has $\pi$ as its invariant distribution. In order to hit a different
probability measure $\eta\in\mathcal{P} ( S ) $, we need to perturb
the original dynamics, which includes the distortion of the Markov chain
transition probability $\alpha$ and the distortion of the exponential holding
time~$\sigma$. Each of these distortions must pay a relative entropy
cost, and
the minimum of the (suitably normalized) sum of these costs asymptotically
approximates the rate function $I ( \eta ) $. When $\eta$ is
singular with respect to $\pi$, the relative entropy cost from the distortion
of $\alpha$ can be made arbitrarily small, and the rate function is almost
entirely due to contributions coming from the distortion of $\sigma$.
We will
illustrate this point via the following example.\vadjust{\goodbreak}

Recall the model mentioned in the \hyperref[intro]{Introduction}, where the state space
$S$ is
$ [ 0,1 ] $, the jump intensity is $q\equiv1$, and for each
$x\in [ 0,1 ] $, $\alpha ( x,\cdot ) $ is the uniform
distribution on $ [ 0,1 ] $. The invariant distribution $\pi$ is
just the uniform distribution on $ [ 0,1 ] $. Now consider a Dirac
measure $\eta\doteq\delta_{1/2}$ as a target measure. $\eta$ is not absolutely
continuous with respect to $\pi$. However, we can approximate $\eta$ weakly
via a sequence of probability measures that are absolutely continuous with
respect to $\pi$. For each $n\in%
\mathbb{N,}
$ define a probability measure $\eta^{n}$ by its Radon--Nikodym derivative
$\theta^{n}$ with respect to $\pi$ according~to
\[
\theta^{n} ( x ) \doteq\cases{ %
n-1, &\quad $\mbox{for }\displaystyle x\in \biggl( \frac{1}{2}-
\frac{1}{2n},\frac{1}{2}+\frac
{1}{2n} \biggr),$
\vspace*{2pt}\cr
\displaystyle\frac{1}{n-1}, &\quad $\mbox{otherwise}.$}
\]
Using the formula (\ref{I1}) for rate function, we have
\[
I \bigl( \eta^{n} \bigr) =1- \biggl( \int_{0}^{1}
\bigl( \theta^{n} ( x ) \bigr) ^{1/2}\,dx \biggr) \biggl( \int
_{0}^{1} \bigl( \theta ^{n} ( y ) \bigr)
^{1/2}\,dy \biggr) =1-\frac{4(n-1)}{n^{2}}.
\]
According to the definition of rate function in Section~\ref
{rate_defn}, the
rate function is bounded above by $1$. However, $I ( \eta^{n} )
\rightarrow1$ as $n\rightarrow\infty$, and one can check that this is
true for
any sequence of absolutely continuous measures converging weakly to~$\eta$.
Thus, $I ( \eta ) =1$.

We now consider fixed $n\in%
\mathbb{N}
$ and examine the perturbed dynamics that can hit the measure $\eta
^{n}$. This
is most easily understood by examining the minimizer in the variational
formula for the rate function, whose form was suggested during the
proof of
the Laplace principle lower bound in Section~\ref{lplb}. Recall that
$\bar{\sigma}_{i} ( \cdot ) $ and $\bar{\alpha}_{i} (
\cdot ) $ are perturbed dynamics for the exponential holding time and
the Markov chain, $\bar{\sigma}_{i} ( \cdot ) $ depends on
$ \{ \bar{X}_{0},\bar{\tau}_{1},\bar{X}_{1},\bar{\tau}_{2},\ldots
,\bar
{X}_{i-1} \} $ and $\bar{\alpha}_{i} ( \cdot ) $ depends on
$ \{ \bar{X}_{0},\bar{\tau}_{1},\bar{X}_{1},\bar{\tau}_{2},\ldots
,\bar
{X}_{i-1},\bar{\tau}_{i} \} $. $\bar{\tau}_{i}$ and $\bar{X}_{i}$ are
chosen recursively according to stochastic kernels $\bar{\sigma
}_{i} (
\cdot ) $ and $\bar{\alpha}_{i} ( \cdot ) $. Specifically,
$\bar{s}_{i}$ is defined by (\ref{soj}) using $\bar{X}_{i}$ and $\bar
{\tau
}_{i}$; $\bar{R}_{T}$ is defined by (\ref{R_T}) using $\bar{s}_{i}$; and
$\bar{\eta}_{T}$ is defined by (\ref{eta_T}) using $\bar{X}_{i}$, $\bar
{\tau
}_{i}$ and $\bar{R}_{T}$. Following the procedure in Section~\ref
{lplb}, we
first define $\mu\in\mathcal{P} ( S\times S ) $ as in (\ref{mu}).
Thus, $\mu$ is the product measure. As before, we use $ [ \mu ]
_{1}$ to denote the first marginal of $\mu$ and $p$ to denote the regular
conditional probability such that $\mu= [ \mu ] _{1}\otimes p$.
Since $\mu$ is a product measure defined by (\ref{mu}), $ [ \mu
]
_{1}$ and $p$ are in fact the same measure and the density with respect to
$\pi$ can be calculated as
%
\begin{equation}
\frac{d [ \mu ] _{1}}{d\pi} ( x ) =\cases{ %
\displaystyle\frac{n}{2}, &\quad $\mbox{for }
\displaystyle x\in \biggl( \frac{1}{2}-
\frac{1}{2n},\frac{1}%
{2}+\frac{1}{2n} \biggr),$
\vspace*{2pt}\cr
\displaystyle\frac{n}{2 ( n-1 ) },&\quad $\mbox{otherwise}.$}\label{marginal}%
\end{equation}
As in Section~\ref{lplb} let $\bar{\alpha}_{i}\doteq p$ for each $i$. A direct
calculation of $A$ using formula (\ref{a}) shows that $A=4 (
n-1 )
/n^{2}$. Also, $\kappa$ defined in (\ref{k}) reduces to
\[
\kappa ( x ) =\cases{ %
\displaystyle\frac{n}{2 ( n-1 ) }, & \quad$\mbox{for }
\displaystyle x\in \biggl( \frac{1}{2}-
\frac
{1}{2n},\frac{1}{2}+\frac{1}{2n} \biggr),$
\vspace*{2pt}\cr
\displaystyle\frac{n}{2}, & \quad $\mbox{otherwise}.$}
\]
As in Section~\ref{lplb}, $\bar{\sigma}_{i}$ should be the exponential
distribution with mean\break $ [ A\kappa ( \bar{X}_{i-1} )  ]
^{-1}$\hspace*{-1pt}. Hence, if $\bar{X}_{i-1}$ falls into $ (
1/2-1/(2n),1/2+1/(2n) ) $, $\bar{\sigma}_{i}$ would be the exponential
distribution with mean $n/2$, otherwise $\bar{\sigma}_{i}$ would be the
exponential distribution with mean $n/ [ 2 ( n-1 )  ] $.
Now the perturbed Markov jump process, denoted by $\bar{X} ( t
) $,
is constructed using $\bar{\alpha}_{i}$ and $\bar{\sigma}_{i}$ defined as
above. As proved in Lemma~\ref{lem:costs}, the expected value of the relative
entropy cost
\[
\frac{1}{T}\sum_{i=1}^{\bar{R}_{T}} \bigl( R
( \bar{\alpha}%
_{i-1}\llVert \alpha ) +R
( \bar{\sigma}_{i}\llVert \sigma ) \bigr)
\]
converges to
\[
I \bigl( \eta^{n} \bigr) =AR \bigl( \mu\llVert [ \mu ] _{1}
\otimes\alpha \bigr) +\int_{0}^{1}\ell
\bigl( A\kappa ( x ) \bigr) \eta^{n} ( dx )
\]
as $T\rightarrow\infty$. We have noted that $p ( x,dy ) = [
\mu ] _{1} ( dy ) $ and $\alpha ( x,dy )
=\pi ( dy ) $, and by using (\ref{marginal})
\begin{eqnarray*}
&& AR \bigl( \mu\llVert [ \mu ] _{1}\otimes\alpha \bigr)
\\
&&\qquad =A\int_{0}^{1}R \bigl( p ( x,\cdot ) \llVert
\alpha ( x,\cdot )  \bigr) [ \mu ] _{1} ( dx )
\\
&&\qquad =\frac{4 ( n-1 ) }{n^{2}} \biggl( \log n-\log2-\frac{\log (
n-1 ) }{2} \biggr).
\end{eqnarray*}
This converges to $0$ as $n\rightarrow\infty$. Hence, the relative
entropy cost
that comes from the distortion of the Markov chain converges to $0$.
For the
second term, we have
\[
\int_{0}^{1}\ell \bigl( A\kappa ( x ) \bigr)
\eta^{n} ( dx ) =\frac{2 ( n-1 ) }{n^{2}} \bigl( \log ( n-1 ) +2\log2-\log n
\bigr) -\frac{4 ( n-1 ) }{n^{2}}+1,
\]
which converges to $1$ as $n\rightarrow\infty$. Thus, as $\eta^{n}$ approaches
the target distribution $\eta$, the relative entropy cost that comes
from the
distortion of Markov chain vanishes, and the rate function becomes solely
determined by the relative entropy cost that comes from the distortion of
exponential waiting times.

One can generalize the argument to more general discrete target measures,
where one utilizes the original dynamics to make sure neighborhoods of the
various points are visited, and then uses the time dilation to control their
relative weight.

\begin{appendix}
\section*{Appendix}\label{app}

\subsection{Proof of inequality (\texorpdfstring{\protect\ref{inequality}}{4.30})}
\mbox{}

\begin{pf}
Recall that $R ( \mu\llVert  [ \mu ] _{1}\otimes
\alpha  ) <\infty$, where $ [ \mu ] _{1}= [
\mu ] _{2}$ and $\widetilde{\pi}$ is invariant under $\alpha$.
Additionally, we also have Condition \ref{cond2'}, that is, there
exists an
integer $N$ and a real number $c\in(0,\infty)$ such that
%
\setcounter{equation}{0}
\begin{equation}
\alpha^{(N)} ( x,\cdot ) \leq c\widetilde{\pi} ( \cdot )
\label{cond4}%
\end{equation}
for all $x\in S$. Now let $p$ be the regular conditional probability
such that
$\mu= [ \mu ] _{1}\otimes p$. Then
\[
R \bigl( \mu\llVert [ \mu ] _{1}\otimes\alpha \bigr) =R \bigl(
[ \mu ] _{1}\otimes p\llVert [ \mu ] _{1}\otimes\alpha
 \bigr) <\infty.
\]
The chain rule of relative entropy implies that
%
\begin{eqnarray}\label{ind}%
&& R \bigl( [ \mu ] _{1}\otimes\mathop{p\otimes\cdots\otimes
p}_{N}\llVert [ \mu ] _{1}\otimes\mathop{\alpha\otimes
\cdots\otimes\alpha}_{N} \bigr)
\nonumber
\\[-8pt]
\\[-8pt]
\nonumber
&&\qquad=N\cdot R \bigl( [ \mu ]
_{1}\otimes p\llVert [ \mu ] _{1}\otimes\alpha \bigr)
<\infty.
\end{eqnarray}
Indeed, since $ [ \mu ] _{1}$ is invariant under $p$, for any
integer $n$ the $n$th marginal of $ [  [ \mu ] _{1}%
\otimes\mathop{p\otimes\cdots\otimes p}_{n-1} ] $ is
\[
\bigl[ [ \mu ] _{1}\otimes\mathop{p\otimes\cdots\otimes
p}_{n-1} \bigr] _{n}= [ \mu ] _{1}.
\]
Hence, (\ref{ind}) follows by induction:
\begin{eqnarray*}
&& R \bigl( [ \mu ] _{1}\otimes\mathop{p\otimes\cdots\otimes
p}_{n}\llVert [ \mu ] _{1}\otimes \mathop{\alpha\otimes\cdots
\otimes\alpha}_{n} \bigr)
\\
&&\qquad =R \bigl( [ \mu ] _{1}\otimes\mathop{p\otimes\cdots\otimes
p}_{n-1}\llVert [ \mu ] _{1}\otimes \mathop{\alpha\otimes
\cdots\otimes\alpha}_{n-1} \bigr) \\
&&\qquad\quad{}+\int_{S}R
( p\llVert \alpha ) \,d \bigl[ [ \mu ] _{1}%
\otimes\mathop{p\otimes\cdots\otimes p}_{n-1} \bigr] _{n}
\\
&&\qquad = ( n-1 ) \cdot R \bigl( [ \mu ] _{1}\otimes p\llVert [ \mu ]
_{1}\otimes\alpha \bigr) +\int_{S}R (
p\llVert \alpha ) \,d [ \mu ] _{1}
\\
&&\qquad =n\cdot R \bigl( [ \mu ] _{1}\otimes p\llVert [ \mu ] _{1}
\otimes\alpha \bigr).
\end{eqnarray*}
Let $ [ \nu ] _{k|j}$ denote the conditional probability of the
$k$th argument of $\nu$ given the $j$th argument of $\nu$. Note that
one can
define a mapping from $\mathcal{P} ( S^{N+1} ) $ to $\mathcal
{P}%
( S^{2} ) $ such that each $\nu\in\mathcal{P} (
S^{N+1} ) $ is mapped to $ [ \nu ] _{1}\otimes [
\nu ] _{N+1|1}$. Since the relative entropy for induced measures is
always smaller, (\ref{ind}) implies
\[
R \bigl( [ \mu ] _{1}\otimes p^{ ( N ) }\llVert [ \mu ]
_{1}\otimes\alpha^{ ( N ) } \bigr) <\infty.
\]
Now since $ [ \mu ] _{1}$ is invariant under $p$, it is also
invariant under $p^{ ( N ) }$ and, therefore, $ [  [
\mu ] _{1}\otimes p^{ ( N ) } ] _{2}= [
\mu ] _{1}$. Using the chain rule of relative entropy again gives
\[
R \bigl( [ \mu ] _{1}\llVert \bigl[ [ \mu ] _{1}\otimes
\alpha^{ ( N ) } \bigr] _{2} \bigr) <\infty.
\]
This implies (\ref{inequality}), since
\begin{eqnarray*}
\infty& >&R \bigl( [ \mu ] _{1}\llVert \bigl[ [ \mu ] _{1}
\otimes\alpha^{ ( N ) } \bigr] _{2} \bigr)
\\
& =&R \bigl( [ \mu ] _{1}\llVert \widetilde{\pi} \bigr) -\log
\int_{S}\frac{d (  [ \mu ] _{1}\otimes
\alpha^{ ( N ) } ) _{2}}{d\widetilde{\pi}} [ \mu ] _{1}
\\
& \geq& R \bigl( [ \mu ] _{1}\llVert \widetilde{\pi} \bigr) -
\log c,
\end{eqnarray*}
where $c$ is from (\ref{cond4}).
\end{pf}

\subsection{Proof of Lemma \texorpdfstring{\protect\ref{rep_ran}}{5.1}}

The proof of the representation is standard, save for the fact that
$R_{T}$ is
random. We include a proof here for completeness.

\begin{pf}
Define for each $k\in%
\mathbb{N}
_{+}$
\[
\eta_{T}^{k} ( \cdot ) \doteq\frac{1}{T} \Biggl[ \sum
_{i=1}%
^{R_{T}\wedge k-1}\delta_{X_{i-1}}
( \cdot ) \frac{\tau_{i}%
}{q ( X_{i-1} ) }+\delta_{X_{R_{T}}\wedge k-1} ( \cdot ) \Biggl( T-\sum
_{i=1}^{R_{T}\wedge k-1}\frac{\tau_{i}}{q (
X_{i-1} )
} \Biggr)
\Biggr].
\]
For any measure $\nu^{k}\in\mathcal{P} (  ( S\times%
\mathbb{R}
_{+} ) ^{k} ) $, we can decompose $\nu^{k}$ as
%
\begin{equation}
\nu^{k}=\bar{\alpha}_{0}\otimes\bar{\sigma}_{1}
\otimes\bar{\alpha}_{1}%
\otimes\bar{\sigma}_{2}
\otimes\cdots\otimes\bar{\alpha}_{k-1}\otimes \bar{\sigma}_{k}.
\label{decom}%
\end{equation}
Choose the barred random variables $\bar{X}_{i}$ and $\bar{\tau}_{i}$
according to $\bar{\alpha}_{i}$ and $\bar{\sigma}_{i}$ as before and define
the corresponding $\bar{R}_{T}\wedge k$ the following way: if $\sum_{i=1}%
^{k}\bar{\tau}_{i}/\break q ( \bar{X}_{i-1} ) >T$, then $\bar
{R}_{T}\wedge
k\doteq\bar{R}_{T}$ where $\bar{R}_{T}$ is the integer that satisfies
\[
\sum_{i=1}^{\bar{R}_{T}-1}\frac{\bar{\tau}_{i}}{q ( \bar
{X}_{i-1} )
}\leq T<
\sum_{i=1}^{\bar{R}_{T}}\frac{\bar{\tau}_{i}}{q ( \bar{X}%
_{i-1} ) },
\]
otherwise define $\bar{R}_{T}\wedge k\doteq k$. We also define
%
\begin{eqnarray}\label{eta_bar_k}%
&&\bar{\eta}_{T}^{k} ( \cdot ) \doteq\frac{1}{T}
\Biggl[ \sum_{i=1}^{\bar{R}_{T}\wedge k-1}\delta_{\bar{X}_{i-1}}
( \cdot ) \frac{\bar{\tau}_{i}}{q ( \bar{X}_{i-1} ) }
\nonumber
\\[-8pt]
\\[-8pt]
\nonumber
&&\hspace*{51pt}{}+\delta_{\bar{X}%
_{\bar{R}_{T}\wedge k-1}} ( \cdot ) \Biggl( T-\sum
_{i=1}^{\bar
{R}_{T}\wedge k-1}\frac{\bar{\tau}_{i}}{q ( \bar{X}_{i-1} )
} \Biggr)
\Biggr].
\end{eqnarray}
If we denote the multidimensional probability measure corresponding to the
original dynamics by $\mu^{k}\in\mathcal{P} (  ( S\times%
\mathbb{R}
_{+} ) ^{k} ) $, that is,
\[
\mu^{k}\doteq\alpha^{ ( k ) }\times \biggl( \prod
_{k}\sigma \biggr) ,
\]
then applying Lemma~\ref{var_for} gives
%
\begin{eqnarray}\label{k_rep1}%
&&-\frac{1}{T}\log E \bigl[ \exp \bigl\{ -TF \bigl( \eta_{T}^{k}
\bigr) \bigr\} \bigr]
\nonumber
\\[-8pt]
\\[-8pt]
\nonumber
&&\qquad =\inf_{\nu^{k}\in\mathcal{P} (  ( S\times%
\mathbb{R}
_{+} ) ^{k} ) } \biggl[ \int
_{ ( S\times%
\mathbb{R}
_{+} ) ^{k}}F \bigl( \bar{\eta}_{T}^{k} \bigr)
\,d\nu^{k}+\frac{1}%
{T}R \bigl( \nu^{k}\llVert
\mu^{k} \bigr) \biggr].
\end{eqnarray}
By applying Theorem~\ref{chain_rule} repeatedly to $R ( \nu^{k}\llVert
\mu^{k}  ) $, we obtain
\[
R \bigl( \nu^{k}\llVert \mu^{k} \bigr) =E \biggl[
\sum_{i=1}%
^{k} \bigl( R (
\bar{\alpha}_{i-1}\llVert \alpha ) +R ( \bar{
\sigma}_{i}\llVert \sigma ) \bigr) \biggr] .
\]
We can thus rewrite (\ref{k_rep1}) as
%
\begin{eqnarray}\label{k_rep2}%
&&-\frac{1}{T}\log E \bigl[ \exp \bigl\{ -TF \bigl( \eta_{T}^{k}
\bigr) \bigr\} \bigr]
\nonumber
\\[-8pt]
\\[-8pt]
\nonumber
&&\qquad=\inf_{\nu^{k}\in\mathcal{P} (  ( S\times%
\mathbb{R}
_{+} ) ^{k} ) }E \Biggl[ F \bigl( \bar{
\eta}_{T}^{k} \bigr) +\frac{1}{T}\sum
_{i=1}^{k} \bigl( R ( \bar{\alpha}_{i-1}
\llVert \alpha ) +R ( \bar{\sigma}_{i}\llVert \sigma
 ) \bigr) \Biggr].
\end{eqnarray}
Now for each $\nu^{k}\in\mathcal{P} (  ( S\times%
\mathbb{R}
_{+} ) ^{k} ) $, we construct another measure $\widehat{\nu}%
^{k}\in\mathcal{P} (  ( S\times%
\mathbb{R}
_{+} ) ^{k} ) $ recursively as follows: define $\widehat
{\alpha
}_{0}\doteq\bar{\alpha}_{0}$ and $\widehat{\sigma}_{1}\doteq\bar{\sigma}_{1}$.
For all $2\leq i\leq k$, define $\widehat{\alpha}_{i-1}$ and $\widehat
{\sigma
}_{i}$ by
\[
( \widehat{\alpha}_{i-1},\widehat{\sigma}_{i} ) =
\cases{ %
( \bar{\alpha}_{i-1},
\bar{\sigma}_{i} ), & \quad$\mbox{if }\displaystyle\sum_{j=1}^{i-1}
\frac{\bar{\tau}_{j}}{q ( \bar{X}_{j-1} ) }\leq T,$
\vspace*{2pt}\cr
( \alpha,\sigma ), & \quad $\mbox{otherwise}.$}
\]
Thus, we return to the original dynamics with zero relative entropy
cost after
$\bar{R}_{T}$. Define $\widehat{\nu}^{k}$ using $\widehat{\alpha}_{i}$ and
$\widehat{\sigma}_{i}$ by (\ref{decom}). From the definition (\ref{eta_bar_k}),
we have $E [ F ( \widehat{\eta}_{T}^{k} )  ] =E [
F ( \bar{\eta}_{T}^{k} )  ] $, and
\begin{eqnarray*}
E \Biggl[ \sum_{i=1}^{k} \bigl( R (
\widehat{\alpha}_{i-1}\llVert \alpha ) +R ( \widehat{
\sigma}_{i}\llVert \sigma ) \bigr) \Biggr] & =&E \biggl[
\sum_{i=1}%
^{\widehat{R}_{T}\wedge k} \bigl( R (
\widehat{\alpha}_{i-1}\llVert \alpha ) +R ( \widehat{
\sigma}_{i}\llVert \sigma ) \bigr) \biggr]
\\
& =&E \Biggl[ \sum_{i=1}^{\bar{R}_{T}\wedge k} \bigl( R
( \bar{\alpha }_{i-1}\llVert \alpha ) +R ( \bar{
\sigma}_{i}\llVert \sigma ) \bigr) \Biggr]
\\
& \leq& E \Biggl[ \sum_{i=1}^{k} \bigl( R
( \bar{\alpha}_{i-1}\llVert \alpha ) +R ( \bar{
\sigma}_{i}\llVert \sigma ) \bigr) \Biggr].
\end{eqnarray*}
Hence, we can rewrite (\ref{k_rep2}) as
%
\begin{eqnarray}\label{k_rep3}%
&&-\frac{1}{T}\log E \bigl[ \exp \bigl\{ -TF \bigl( \eta_{T}^{k}
\bigr) \bigr\} \bigr]
\nonumber
\\[-8pt]
\\[-8pt]
\nonumber
&&\qquad=\inf_{\nu^{k}\in\mathcal{P} (  ( S\times%
\mathbb{R}
_{+} ) ^{k} ) }E \Biggl[ F \bigl( \bar{
\eta}_{T}^{k} \bigr) +\frac{1}{T}\sum
_{i=1}^{\bar{R}_{T}\wedge k} \bigl( R ( \bar{\alpha
}_{i-1}\llVert \alpha ) +R ( \bar{
\sigma}_{i}\llVert \sigma ) \bigr) \Biggr].
\end{eqnarray}
Using the pointwise convergence of both $R_{T}\wedge k\rightarrow
R_{T}$ and
$\bar{R}_{T}\wedge k\rightarrow\bar{R}_{T}$ as $k\rightarrow\infty$, by the
dominated convergence theorem
\begin{eqnarray*}
\lim_{k\rightarrow\infty}-\frac{1}{T}\log E \bigl[ \exp \bigl\{ -TF
\bigl( \eta_{T}^{k} \bigr) \bigr\} \bigr] &=&-\frac{1}{T}
\log E \bigl[ \exp \bigl\{ -TF ( \eta_{T} ) \bigr\} \bigr],
\\
\lim_{k\rightarrow\infty}E \bigl[ F \bigl( \bar{\eta
}_{T}^{k} \bigr) \bigr] &=& E \bigl[ F ( \bar{
\eta}_{T} ) \bigr].
\end{eqnarray*}
Also, by the monotone convergence theorem
\[
\lim_{k\rightarrow\infty}E \Biggl[ \sum_{i=1}^{\bar{R}_{T}\wedge k}
\bigl( \bigl( R ( \bar{\alpha}_{i-1}\llVert \alpha )
+R ( \bar{\sigma}_{i}\llVert \sigma ) \bigr) \bigr)
\Biggr] =E \Biggl[ \sum_{i=1}^{\bar{R}_{T}} \bigl(
\bigl( R ( \bar{\alpha}_{i-1}\llVert \alpha ) +R (
\bar{\sigma }_{i}\llVert \sigma ) \bigr) \bigr) \Biggr].
\]
Hence, by taking limits on both sides of (\ref{k_rep3}), we arrive at
\[
-\frac{1}{T}\log E \bigl[ \exp \bigl\{ -TF ( \eta_{T} ) \bigr\}
\bigr] =\inf\bar{E} \Biggl[ F ( \bar{\eta}_{T} ) +
\frac{1}%
{T}\sum_{i=1}^{\bar{R}_{T}} \bigl(
\bigl( R ( \bar{\alpha}%
_{i-1}\llVert \alpha
) +R ( \bar{\sigma}_{i}\llVert \sigma ) \bigr)
\bigr) \Biggr],
\]
where the infimum is taken over all controlled measures $ \{ \bar
{\alpha
}_{i},\bar{\sigma}_{i} \} $. This proves the lemma.
\end{pf}
\end{appendix}
\section*{Acknowledgments}
The authors thank a referee for useful comments
and a correction.

%


\printaddresses


\begin{thebibliography}{11}

\bibitem{chelu}
%
\begin{barticle}[mr]
\bauthor{\bsnm{Chen},~\bfnm{Mu~Fa}\binits{M.~F.}} \AND
\bauthor{\bsnm{Lu},~\bfnm{Yun~Gang}\binits{Y.~G.}}
(\byear{1990}).
\btitle{On evaluating the rate function of large deviations for jump processes}.
\bjournal{Acta Math. Sinica (N.S.)}
\bvolume{6}
\bpages{206--219}.
\bid{doi={10.1007/BF02108200}, issn={1000-9574}, mr={1071599}}
\end{barticle}\vspace*{-9pt}
%
\bptok{imsref}%
\endbibitem

\bibitem{donvar1}
%
\begin{barticle}[mr]
\bauthor{\bsnm{Donsker},~\bfnm{M.~D.}\binits{M.~D.}} \AND
\bauthor{\bsnm{Varadhan},~\bfnm{S.~R.~S.}\binits{S.~R.~S.}}
(\byear{1975}).
\btitle{Asymptotic evaluation of certain {M}arkov process expectations
for large time. {I}}.
\bjournal{Comm. Pure Appl. Math.}
\bvolume{28}
\bpages{1--47}.
\bid{issn={0010-3640}, mr={0386024}}
\end{barticle}
%
\bptok{imsref}%
\endbibitem

\bibitem{donvar3}
%
\begin{barticle}[mr]
\bauthor{\bsnm{Donsker},~\bfnm{M.~D.}\binits{M.~D.}} \AND
\bauthor{\bsnm{Varadhan},~\bfnm{S.~R.~S.}\binits{S.~R.~S.}}
(\byear{1976}).
\btitle{Asymptotic evaluation of certain {M}arkov process expectations
for large time. {III}}.
\bjournal{Comm. Pure Appl. Math.}
\bvolume{29}
\bpages{389--461}.
\bid{issn={0010-3640}, mr={0428471}}
\end{barticle}
%
\bptok{imsref}%
\endbibitem

\bibitem{dupell4}
%
\begin{bbook}[mr]
\bauthor{\bsnm{Dupuis},~\bfnm{Paul}\binits{P.}} \AND
\bauthor{\bsnm{Ellis},~\bfnm{Richard~S.}\binits{R.~S.}}
(\byear{1997}).
\btitle{A Weak Convergence Approach to the Theory of Large Deviations}.
\bpublisher{Wiley},
\blocation{New York}.
\bid{doi={10.1002/9781118165904}, mr={1431744}}
\end{bbook}
%
\bptok{imsref}%
\endbibitem

\bibitem{ellwyn}
%
\begin{barticle}[mr]
\bauthor{\bsnm{Ellis},~\bfnm{Richard~S.}\binits{R.~S.}} \AND
\bauthor{\bsnm{Wyner},~\bfnm{Aaron~D.}\binits{A.~D.}}
(\byear{1989}).
\btitle{Uniform large deviation property of the empirical process of a
{M}arkov chain}.
\bjournal{Ann. Probab.}
\bvolume{17}
\bpages{1147--1151}.
\bid{issn={0091-1798}, mr={1009449}}
\end{barticle}
%
\bptok{imsref}%
\endbibitem

\bibitem{ethkur}
%
\begin{bbook}[mr]
\bauthor{\bsnm{Ethier},~\bfnm{Stewart~N.}\binits{S.~N.}} \AND
\bauthor{\bsnm{Kurtz},~\bfnm{Thomas~G.}\binits{T.~G.}}
(\byear{1986}).
\btitle{Markov Processes: Characterization and Convergence}.
\bpublisher{Wiley},
\blocation{New York}.
\bid{doi={10.1002/9780470316658}, mr={0838085}}
\end{bbook}
%
\bptok{imsref}%
\endbibitem

\bibitem{fol}
%
\begin{bbook}[mr]
\bauthor{\bsnm{Folland},~\bfnm{Gerald~B.}\binits{G.~B.}}
(\byear{1999}).
\btitle{Real Analysis: Modern Techniques and their Applications},
\bedition{2nd} ed.
\bpublisher{Wiley},
\blocation{New York}.
\bid{mr={1681462}}
\end{bbook}
%
\bptok{imsref}%
\endbibitem

\bibitem{liu2}
%
\begin{bmisc}[auto:STB|2014/02/12|14:17:21]
\bauthor{\bsnm{Liu},~\bfnm{Y.}\binits{Y.}}
(\byear{2013}).
\bhowpublished{Large deviations rate functions for the empirical
measure: Explicit formulas and an application to Monte Carlo.
Ph.D. thesis, Brown Univ., Providence, RI.}
\end{bmisc}
%
\bptok{imsref}%
\endbibitem

\bibitem{pin}
%
\begin{barticle}[mr]
\bauthor{\bsnm{Pinsky},~\bfnm{Ross}\binits{R.}}
(\byear{1985}).
\btitle{On evaluating the {D}onsker--{V}aradhan {$I$}-function}.
\bjournal{Ann. Probab.}
\bvolume{13}
\bpages{342--362}.
\bid{issn={0091-1798}, mr={0781409}}
\end{barticle}
%
\bptok{imsref}%
\endbibitem

\bibitem{rud}
%
\begin{bbook}[mr]
\bauthor{\bsnm{Rudin},~\bfnm{Walter}\binits{W.}}
(\byear{1991}).
\btitle{Functional Analysis},
\bedition{2nd} ed.
\bpublisher{McGraw-Hill},
\blocation{New York}.
\bid{mr={1157815}}
\end{bbook}
%
\bptok{imsref}%
\endbibitem

\bibitem{str1}
%
\begin{bbook}[mr]
\bauthor{\bsnm{Stroock},~\bfnm{D.~W.}\binits{D.~W.}}
(\byear{1984}).
\btitle{An Introduction to the Theory of Large Deviations}.
\bpublisher{Springer},
\blocation{New York}.
\bid{doi={10.1007/978-1-4613-8514-1}, mr={0755154}}
\end{bbook}
%
\bptok{imsref}%
\endbibitem

\end{thebibliography}
\end{document}